\documentclass[oneside,reqno,10pt]{amsart}
\usepackage{graphicx,amsfonts,amssymb,amsmath,amsthm,url}
\usepackage{color}
\usepackage[usenames,dvipsnames]{xcolor}
\usepackage[normalem]{ulem}
\usepackage{pdfsync}

\usepackage[hyperfootnotes=false, colorlinks, citecolor=	 RoyalBlue, urlcolor=blue, linkcolor=blue         ]{hyperref}

\theoremstyle{plain} %--default
\newtheorem{theorem}             {Theorem}  % [section]
\newtheorem{lemma}      [theorem]{Lemma}

\newtheorem{proposition}[theorem]{Proposition}

\theoremstyle{definition}

\theoremstyle{remark}
\newtheorem{remark}              {Remark}

\renewcommand{\geq}{\geqslant}
\renewcommand{\leq}{\leqslant}

\DeclareMathOperator{\SL}{SL}
\DeclareMathOperator{\Mp}{Mp}
\DeclareMathOperator{\GL}{GL}

\DeclareMathOperator{\htt}{ht}

\DeclareMathOperator{\SO}{SO}
\DeclareMathOperator{\ad}{ad}
\DeclareMathOperator{\Ad}{Ad}

\DeclareMathOperator{\Hom}{Hom}

\def\eps{\varepsilon}

\DeclareMathOperator{\PGL}{PGL}

\DeclareMathOperator{\Sob}{Sob}

\DeclareMathOperator{\diag}{diag}

\DeclareMathOperator{\Lie}{Lie}

\def\O{\operatorname{O}}

\DeclareMathOperator{\sym}{sym}

\DeclareMathOperator{\cusp}{cusp}

\DeclareMathOperator{\Eis}{Eis}
\DeclareMathOperator{\reg}{reg}

\DeclareMathOperator{\Sp}{Sp}

\DeclareMathOperator{\vol}{vol}

\renewcommand{\Re}{\mathrm{Re}}
\renewcommand{\Im}{\mathrm{Im}}
\def\eps{\varepsilon}

\begin{document}

\title{The spectral decomposition of $|\theta|^2$}

\author{Paul D. Nelson}
\email{paul.nelson@math.ethz.ch}
\address{ETH Zurich, Department of Mathematics, R{\"a}mistrasse 101, CH-8092, Zurich, Switzerland}

\begin{abstract}
  Let $\theta$ be an elementary theta function, such as the
  classical Jacobi theta function.  We establish a spectral
  decomposition and surprisingly strong asymptotic formulas for
  $\langle |\theta|^2, \varphi \rangle$ as $\varphi$ traverses a
  sequence of Hecke-translates of a nice enough fixed
  function. The subtlety is that typically
  $|\theta|^2 \notin L^2$.
  Applications
  to the subconvexity, quantum variance
  and $4$-norm problems
  are indicated.
\end{abstract}
\maketitle

\setcounter{tocdepth}{1}
%\tableofcontents

\section{Introduction}
\label{sec-1}

Let $\Gamma$ be a congruence subgroup of the modular group
$\SL_2(\mathbb{Z})$, and let $\mathbb{H}$ denote the upper
half-plane.
For pair of \emph{square-integrable} automorphic functions
$\varphi_1, \varphi_2 \in L^2(\Gamma \backslash \mathbb{H})$,
the Petersson
inner product
\begin{equation}\label{eq:petersson-inner-product-defn}
  \langle \varphi_1, \varphi_2 \rangle
:=
\frac{1}{\vol(\Gamma \backslash \mathbb{H})}
\int_{z \in \Gamma \backslash \mathbb{H}}
\overline{\varphi_1}(z) \varphi_2(z)\, \frac{d x \, d y}{y^2}
\end{equation}
may be written in terms of the inner products of $\varphi$
against the constant function $1$, the elements of an
orthonormal basis $\mathcal{B}_{\cusp}$ for the space of cusp
forms, and the unitary Eisenstein series
$E_{\mathfrak{a},1/2+it}$ attached to the various cusps
$\mathfrak{a}$ of $\Gamma$: with suitable normalization
(see, e.g., \cite{MR1942691},
\cite[\S15]{MR2061214}, \cite{MR546600} for details),
\begin{equation}\label{eq:plancherel-standard}
  \begin{split}
    \langle \varphi_1, \varphi_2 \rangle
    &= \langle \varphi_1, 1 \rangle \langle 1, \varphi_2 \rangle
    + \sum_{\varphi \in \mathcal{B}_{\cusp}}
    \langle \varphi_1, \varphi  \rangle
    \langle \varphi, \varphi_2 \rangle
    \\
    &\quad +
    \sum_{\mathfrak{a}}
    \int_{t \in \mathbb{R}}
    \langle \varphi_1, E_{\mathfrak{a},1/2+i t} \rangle
    \langle E_{\mathfrak{a},1/2+i t}, \varphi_2 \rangle
    \, \frac{d t}{4 \pi}.
  \end{split}
\end{equation}
This formula
may be used to establish (among other things) the equidistribution of the Hecke
correspondences $T_n$ on $\Gamma \backslash \mathbb{H}$ through
estimates such as
\begin{equation}\label{eq:hecke-equid}
  \langle \varphi_1, T_n \varphi_2 \rangle
  = 
  \langle \varphi_1, 1 \rangle \langle 1, \varphi_2 \rangle
  + O(\tau(n) n^{-1/2+\vartheta})
\end{equation}
for unit vectors $\varphi_1, \varphi_2 \in L^2(\Gamma \backslash
\mathbb{H})$.
Here the Hecke operator $T_n$ is normalized so that $T_n 1 = 1$
(thus for $\Gamma = \SL_2(\mathbb{Z})$,
$T_n f(z)$ is the \emph{average}
of $f((a z +b)/d)$ over
all factorizations $n = a d$ and integers $0 \leq b < d$)
and $\vartheta \in [0, 7/64]$ quantifies
the strongest known bound \cite{MR1937203} for the Hecke
eigenvalues of
the cusp forms and Eisenstein series occurring in
\eqref{eq:plancherel-standard}.
Given such a bound for Hecke eigenvalues,
the estimate
\eqref{eq:hecke-equid} follows from
\eqref{eq:plancherel-standard} and
the Cauchy--Schwarz inequality.

Variations on \eqref{eq:hecke-equid} in which
$\varphi_1,\varphi_2$ are \emph{not} both square-integrable
turn out to play
a fundamental role in analytic number theory, extending far
beyond the evident application to Hecke equidistribution.  For
instance,
in the periods-based approach to the subconvexity problem on
$\GL_2$ following
Venkatesh \cite{venkatesh-2005} and
Michel--Venkatesh \cite{michel-2009}, the basic quantitative
input is an analogue of \eqref{eq:hecke-equid} for
\begin{itemize}
\item  $\varphi_1 = |E|^2$ the squared magnitude of a
  unitary Eisenstein series $E$
  and
\item  $\varphi_2 = |\Phi|^2$
  that of a cusp form $\Phi$,
\end{itemize}
so that $\varphi_2$ is rapidly-decaying but $\varphi_1$ fails to
belong to $L^2$, or even to $L^1$.
The inner products $\langle |E|^2, |\Phi|^2 \rangle$ arise
naturally
after applying Cauchy--Schwarz to the integral
representation
$L(\Psi \times \Phi,1/2) = \langle \Psi \Phi, E \rangle$
for the Rankin--Selberg $L$-function
 attached to a pair cusp forms
$\Psi,\Phi$;
the magnitudes of such $L$-functions
are in turn
related to fundamental arithmetic equidistribution problems,
such as those concerning the distribution of solutions
to $x^2  + y^2 + z^2 = n$
(see for instance \cite{MichelVenkateshICM}).
The standard Plancherel formula does not apply
to $\langle |E|^2, |\Phi|^2 \rangle$,
and indeed, its ``formal application'' gives the wrong answer.
There arises the need for a \emph{regularized Plancherel
  formula}
which the authors of \cite[\S4.3.8]{michel-2009}
develop in generality sufficient for their purposes.

% The magnitude of
% the Rankin--Selberg $L$-functions $L(\Psi \times \Phi,1/2)$
% attached to a pair of cusp forms $\Psi \times \Phi$ is directly
% related to many arithmetic equidistribution problems,
% such as those concerning the distribution of solutions
% to $x^2  + y^2 + z^2 = n$
% (see for instance ???).
% The inner products $\langle |E|^2, |\Phi|^2 \rangle$
% arise when studying
% that magnitude
% after an application of
% Cauchy--Schwarz to
% the integral
% representation
% $L(\Psi \times \Phi,1/2)
% = \langle \Psi \Phi, E \rangle$.

% are of interest for studying
% that matnigutde

% in studying

% the Rankin--Selberg $L$-functions
% $L(\Psi \times \Phi,1/2)$
% attached to a pair of cusp forms
% $\Psi \times \Phi$;
% they arise after applying
% Cauchy--Schwarz to
% the Rankin--Selberg integral
% representation
% $L(\Psi \times \Phi,1/2)
% = \langle \Psi \Phi, E \rangle$

% arise after applying Cauchy--Schwarz to
% the Rankin--Selberg integral
% representation
% $L(\Psi \times \Phi,1/2)
% = \langle \Psi \Phi, E \rangle$
% for a pair of cusp forms
% $\Psi \times \Phi$.

This paper is concerned with another such variation.
We will
prove analogues (in fact, counterintuitive
\emph{strengthenings}) of
the spectral decomposition \eqref{eq:plancherel-standard} and
the asymptotic formula
\eqref{eq:hecke-equid} for
\begin{itemize}
\item  $\varphi_1 := |\theta|^2$ the squared
  magnitude of
  an elementary theta function,
  such as
  the \emph{Jacobi theta function}
  \begin{equation}\label{eq:jacobi-theta-defn-booyah}
    \theta(z) := y^{1/4} \sum_{n \in \mathbb{Z}} \exp(2 \pi i n^2 z), \quad z = x + i y
  \end{equation}
  on $\Gamma_0(4) \supseteq \Gamma$, and
\item  $\varphi_2 =: \varphi$ of
  sufficiently mild growth that the inner product
  $\langle |\theta|^2, \varphi  \rangle$
  converges
  absolutely;
  for instance,
  it will suffice to impose
  the growth condition
  \begin{equation}\label{eq:assumption-decay}
    \varphi(z) \ll \htt(z)^{1/2-\delta}
  \end{equation}
  for some fixed $\delta > 0$, where
  the height function
  $\htt : \Gamma \backslash \mathbb{H} \rightarrow \mathbb{R}_+$
  is defined by $\htt(z) := \max_{\gamma \in \SL_2(\mathbb{Z})} \Im(\gamma
  z)$. 
  % The
  % estimates $|\theta(z)| \ll \htt(z)^{1/2}$ and
  % $\int_{y \geq 1} y^{1 - \delta} \, \frac{d y}{y^2}<\infty$
  % then imply
  % the absolute convergence
  % of the inner product $\langle |\theta|^2, \varphi \rangle$.
\end{itemize}
As we discuss below, a detailed analysis of such inner products
$\langle |\theta|^2, \varphi\rangle$ is at the heart of each of
the recent works \cite{nelson-subconvex-reduction-eisenstein,
  nelson-variance-73-2, nelson-variance-II, nelson-variance-3},
and also seems likely to be useful in further applications,
motivating the focused discussion recorded here.

Unfortunately, $|\theta|^2 \notin L^2$, so the standard
Plancherel formula does not apply to the inner products
$\langle |\theta|^2, \varphi \rangle$.  Indeed, it is immediate
from \eqref{eq:petersson-inner-product-defn} that a continuous
function $f$ on $\Gamma \backslash \mathbb{H}$ satisfying the
asymptotic $|f(z)| \asymp \htt(z)^\beta$ near the cusps is
absolutely integrable if and only if $\beta < 1$,
while from \eqref{eq:jacobi-theta-defn-booyah} we see that
$|\theta|^4(z) \gg \htt(z)$ near the cusp $\infty$.
In this article,
we develop a different, robust
technique for decomposing and estimating
such inner products.
We
focus for now on the following
special case of the results obtained,
postponing general statements
to \S\ref{sec:intro-main-results}.
\begin{theorem}\label{thm:hecke-equid-theta}
  Let $\theta$ denote the Jacobi theta function \eqref{eq:jacobi-theta-defn-booyah}.
  Fix a measurable function
  $\varphi$ on $\Gamma \backslash \mathbb{H}$
  satisfying the growth condition \eqref{eq:assumption-decay}.
  % Let $\varphi \in L^2(\Gamma \backslash \mathbb{H})$
  % be fixed and
  % satisfy
  % \begin{equation}\label{eq:assumption-decay}
  %   \varphi(z) \ll \htt(z)^{1/2-\delta}
  %   \text{ for some } \delta > 0
  % \end{equation}
  % with $\htt(z) := \max_{\gamma \in \Gamma} \Im(\gamma z)$,
  % so that $\langle |\theta|^2, \varphi  \rangle$ converges absolutely.
  Let $n$ traverse a sequence of integers coprime to the level of $\Gamma$.
  Then
  \begin{equation}\label{eq:hecke-equid-theta}
    \langle |\theta|^2, T_n \varphi \rangle
    = \langle |\theta|^2, 1 \rangle \langle 1, \varphi \rangle
    + O(\tau(n) n^{-1/2}).
  \end{equation}
\end{theorem}
% We develop here a different technique for decomposing
% $\langle |\theta|^2, \varphi \rangle$.

Curiously, the new bound
\eqref{eq:hecke-equid-theta} is \emph{stronger}
than the more straightforward estimate
\eqref{eq:hecke-equid}
in that $n^{-1/2}$ replaces $n^{-1/2+\vartheta}$.
To put it another way,
the strength of \eqref{eq:hecke-equid-theta}
is comparable to that of
\eqref{eq:hecke-equid} together with the assumption
of the (unsolved) Ramanujan conjecture $\vartheta = 0$.
We will explain this surprise shortly.
% TODO: perhaps just explain the surprise here, indicating
% how it follows from the fact that $|\theta|^2$ is orthogonal
% to cusp forms.

% TODO: insert section heading for applications?

Before doing so, we summarize
our interest in
studying
inner products involving $|\theta|^2$.
% involving
% their squared magnitudes
Our original motivation was that asymptotic formulas such as
\eqref{eq:hecke-equid-theta}
% for such quantities
turned out to be the global quantitative inputs to the method
introduced and developed in \cite{nelson-variance-73-2,
  nelson-variance-II, nelson-variance-3} for attacking the
\emph{quantum variance} problem.
% (see especially \S5.4,
% \S9 of that article).
That problem concerns
the sums
given by
\begin{equation}\label{eq:qv-sums}
  \sum_{f \in \mathcal{F}}
\langle |f|^2, \Psi_1 \rangle
\langle \Psi_2, |f|^2 \rangle
\end{equation}
for some ``nice enough family of automorphic
forms'' $\mathcal{F}$
and fixed observables $\Psi_1,\Psi_2$,
which we assume here for simplicity to be cuspidal.
The asymptotic determination
of \eqref{eq:qv-sums}
when $\mathcal{F}$ consists
of the Maass forms of eigenvalue bounded by some parameter $T
\rightarrow \infty$
may be understood as a fundamental problem in semiclassical analysis
(see for instance \cite[\S15.6]{2009arXiv0911.4312Z},
\cite[\S4.1.3]{MR3204186},
\cite{MR2651907}
\cite{MR2103474},
\cite[\S1]{nelson-variance-73-2}).
The method
of \cite{nelson-variance-73-2}
uses the theta correspondence
to relate
the sums \eqref{eq:qv-sums}
to inner products roughly of the form
\begin{equation}\label{eq:qv-inner-product-theta-appl}
  \langle |\theta|^2, h_1 \overline{h_2} \rangle,
\end{equation}
where $h_1,h_2$ are half-integral weight
Maass--Shimura--Shintani--Waldspurger lifts of $\Psi_1,\Psi_2$.
The local data underlying the lifts depends rather delicately
upon the family $\mathcal{F}$.  It turns out
that when $\mathcal{F}$ is ``nice enough,'' the product
$h_1 \overline{h_2}$ is essentially a \emph{translate} of the
product $\varphi := h_1^0 \overline{h_2^0}$ of some \emph{fixed}
lifts $h_1^0, h_2^0$.  If that translate is induced by (for
instance) the correspondence $T_n$, then
\eqref{eq:qv-inner-product-theta-appl} is of the form
$\langle |\theta|^2, T_n \varphi \rangle$ and so estimates like
\eqref{eq:hecke-equid-theta} become relevant for determining the
asymptotics of sums like \eqref{eq:qv-sums}.

Another motivation
for the present study comes from the appearance of elementary
theta functions in the Shimura
integral representation \cite{Sh75,MR533066}
\begin{equation}\label{eqn:shimura-integral-sketch}
    L(\sym^2 \varphi, 1/2) \approx \int \varphi \theta \tilde{E}
\end{equation}
for symmetric square $L$-functions on $\GL_2$; here $\varphi$ is
a cusp form and $\tilde{E}$ is a suitable half-integral weight
Eisenstein series, and $\approx$ denotes equality
up to
local factors.
One also knows the cuspidal analogue
of \eqref{eqn:shimura-integral-sketch},
namely
\begin{equation}\label{eqn:qiu-integral-abbrev}
    L(\sym^2 \varphi \otimes \Psi, 1/2) \approx |\int \varphi
  \theta h |^2
\end{equation}
for cusp forms $\Psi$ with half-integral lift $h$
(see Qiu \cite[Thm 4.5]{MR3291638}).
An application of the Cauchy--Schwartz inequality to such
integrals yields inner products involving $|\theta|^2$.
By
studying those inner products using the results of this paper, a
surprising implication concerning the subconvexity problem for
(twisted) symmetric square $L$-functions was obtained in
\cite{nelson-subconvex-reduction-eisenstein}.

A third source of motivating applications may be obtained by
summing (precise forms of) the identity
\eqref{eqn:qiu-integral-abbrev} over either $\varphi$ or $h$ in
an orthonormal basis.  The LHS of the resulting identity is a
moment of $L$-functions, while the RHS, by Parseval, is an inner
product involving $|\theta|^2$.  By applying the results of this
paper to such inner products, one may hope to obtain summation
formulas for certain moments of $L$-functions.  Summing over
$\varphi$ yields the moments relevant for the quantum variance
problem discussed above, while taking a suitably normalized sum
over $h$ yields a summation formula for twisted symmetric square
$L$-functions to which one might hope to apply the techniques of
\cite{nelson-variance-II}.  We plan to pursue this idea
separately.

For the applications motivating this work, it is indispensible
to have more flexible forms
of Theorem \ref{thm:hecke-equid-theta}
% \ref{thm:hecke-equid-theta} and Theorem \ref{thm:theta-expand}
in which:
\begin{itemize}
\item The squared magnitude $|\theta|^2$
  is generalized
  to any product $\theta_1 \overline{\theta_2}$
  of unary theta series $\theta_1, \theta_2$, such as those
  obtained by imposing
  congruence conditions in the summation defining
  $\theta$.
\item One allows % in Theorem \ref{thm:hecke-equid-theta}
  greater variation than $\varphi \mapsto T_n \varphi$ for $n$
  coprime to the level of $\Gamma$.  For instance, one would
  like to consider variation under the diagonal flow on
  $\Gamma \backslash \SL_2(\mathbb{R})$, or with respect to
  ``Hecke operators'' at primes dividing the level.
\item The dependence of the error term upon
  $\varphi$ and $\theta_1, \theta_2$ is quantified.
\end{itemize}
The main purpose of this article
is to supply such flexible forms.
Our results,
to be formulated precisely
in \S\ref{sec:intro-main-results},
are natural completions
of Theorem \ref{thm:hecke-equid-theta}:
working over a fixed global field,
we prove that the translates
under the metaplectic group
of a product of two elementary
theta functions
equidistribute
with an essentially optimal rate and polynomial dependence
upon all parameters.
In quantifying
the latter
we make systematic use
of the adelic Sobolev norms
developed in \cite[\S2]{michel-2009}.

% Beyond the motivating application
% to the quantum variance problem in its many unexplored aspects,
% the expansions and estimates established in this article
% seem ripe for further
% exploitation,
% owing for instance to the appearance
% of $\theta$ 
% in the Shimura integral.
% They also
% appear to be of intrinsic interest.

The regularized
Plancherel formula of Michel--Venkatesh,
which builds on a method of Zagier \cite{MR656029},
does not
seem to apply to
the inner products $\langle |\theta|^2, \varphi  \rangle$
considered here:
that method would involve finding
an Eisenstein series
$\mathcal{E}$ with parameter
to the right of the unitary axis
for which $|\theta|^2 - \mathcal{E} \in L^2$,
but
it is easy to see from the  
expansion $|\theta|^2(z) \sim y^{1/2} + \dotsb$
near the cusp $\infty$ 
that such an $\mathcal{E}$ does not exist.
% for the unitary axis.  The basic
% method This may be
% ,\footnote{ , which invalidates the
%   hypotheses of \cite[\S4.3.8]{michel-2009} as well as the basic
%   method: there is no non-unitary Eisenstein series
%   $\mathcal{E}$ for which $|\theta|^2 - \mathcal{E} \in L^2$.
%   However, see Remark \ref{rmk:other-proofs}.  } and
The singular nature of the parameter $1/2$
presents further difficulties
if one tries to adapt that method.
We are not
% The regularized
% Plancherel formula of Michel--Venkatesh likewise does not
% apply,\footnote{ In the asymptotic expansion
%   $|\theta|^2(z) \sim y^{1/2} + \dotsb$ near the cusps, the
%   exponent $1/2$ lies on the unitary axis, which invalidates the
%   hypotheses of \cite[\S4.3.8]{michel-2009} as well as the basic
%   method: there is no non-unitary Eisenstein series
%   $\mathcal{E}$ for which $|\theta|^2 - \mathcal{E} \in L^2$.
%   However, see Remark \ref{rmk:other-proofs}.  } and we are not
aware of an adequate formal reduction (e.g., via a simple
approximation argument or truncation) by which one may deduce
flexible forms of estimates such as \eqref{eq:hecke-equid-theta}
directly from \eqref{eq:plancherel-standard}.
The technique developed here is
more direct, and specific to $|\theta|^2$.
% ;
% it was found after some searching (see Remark \ref{rmk:other-proofs}).
% We develop instead a different, robust technique for decomposing
% $\langle |\theta|^2, \varphi \rangle$
% (after much searching; see Remark \ref{rmk:other-proofs}).
We illustrate it now briefly
in the context of Theorem
\ref{thm:hecke-equid-theta}:
% We now briefly indicate
% the steps leading to the proof of Theorem
% \ref{thm:hecke-equid-theta}:
\begin{enumerate}
\item Some careful but elementary manipulations (change of
  variables, Poisson summation, folding up, Mellin inversion, contour shift)
  to be explained in \S\ref{sec:basic-idea}
  give a
  pointwise expansion
  \begin{equation}\label{eq:theta-expand-pointwise}
    |\theta|^2
    =
    1
    + \sum_{\mathfrak{a}}
    \int_{t \in \mathbb{R}}
    c(\mathfrak{a},t)
    E_{\mathfrak{a},1/2+it}^*
    \, \frac{d t}{2 \pi }
  \end{equation}
  where $\mathfrak{a}$ traverses the cusps of $\Gamma_0(4)$,
  \[
    E_{\mathfrak{a},1/2+it}^* := 2 \xi(1 + 2 i t)
    E_{\mathfrak{a}, 1/2+ i t},
    \quad
    \xi(s) := \pi^{-s/2} \Gamma(s/2) \zeta(s),
  \]
  and the complex coefficients $c(\mathfrak{a},t)$ are explicit
  and uniformly bounded (and best described in the language of
  Theorem \ref{thm:theta-plancherel-redux} below; see also
  \cite[\S9]{nelson-variance-73-2}).
  It seems worth recalling here
  that the standard
  unitary Eisenstein series $E_{\mathfrak{a},1/2+i t}$,
  as
  given
  in the simplest case $\mathfrak{a} = \infty, \Gamma =
  \SL_2(\mathbb{Z})$ by
  \begin{equation}\label{eq:defn-classical-eis-series}
      E_s(z)
  :=
  \frac{1}{2} 
  \sum_{
    \substack{
      (c,d) \in \mathbb{Z}^2 - \{(0,0)\} : \\
      \gcd(c,d) = 1
    }
  }
  \frac{y^s}{ |c z + d|^{2 s}}
  =
  y^s +
  \frac{\xi(2 s - 1)}{\xi(2 s)} y^{1-s} + \dotsb
  \end{equation}
  for
  $\Re(s) > 1$,
  vanishes
  like $O(t)$ as $t \rightarrow 0$, while $\xi(1 + 2 it)$ has a
  simple pole at $t = 0$, so the normalized variant
  $E_{\mathfrak{a}, 1/2+i t}^*$ is well-defined and
  $E_{\mathfrak{a},1/2}^*$ is not identically zero.
\item
  From the expansion \eqref{eq:theta-expand-pointwise}
  we deduce first that
  $\langle |\theta|^2, 1  \rangle = 1$
  and then that
  \begin{equation}\label{eq:theta-expand}
    \langle |\theta|^2, \varphi  \rangle
    = \langle |\theta|^2, 1 \rangle
    \langle 1, \varphi \rangle
    + \sum_{\mathfrak{a}}
    \int_{t \in \mathbb{R}}
    c(\mathfrak{a},t)
    \langle E_{\mathfrak{a},1/2+it}^*, \varphi \rangle
    \, \frac{d t}{2 \pi }.
  \end{equation}
\item
  We deduce Theorem \ref{thm:hecke-equid-theta}
  in the expected way
  by
  replacing $\varphi$ with $T_n \varphi$
  and
  appealing to the consequence
  $\langle E_{\mathfrak{a},1/2+it}^*, T_n \varphi \rangle
  \ll (1 + |t|)^{-100} \tau(n) n^{-1/2}$
  of standard bounds for unitary Eisenstein
  series
  and the rapid decay of $\xi(1+2 i t)$.
\end{enumerate}

The main analytic difference between the integrands on
the RHS of
\eqref{eq:plancherel-standard} and \eqref{eq:theta-expand} is in
their $t \rightarrow 0$ behavior: 
for nice enough $\varphi,\varphi_1,\varphi_2$,
the typical magnitude of the integrand in
\eqref{eq:plancherel-standard}
is $\asymp |t|$ while that in \eqref{eq:theta-expand} is
$\asymp 1$.
The more glaring
difference between
the two expansions is that
\eqref{eq:theta-expand} contains no
cuspidal contribution.
The
improvement
of \eqref{eq:hecke-equid-theta}
compared to
\eqref{eq:hecke-equid}
is now explained by noting as above that
the Hecke eigenvalues of
unitary Eisenstein series
(unlike those of cusp forms)
are known to
satisfy bounds consistent with the
Ramanujan conjecture
$\vartheta = 0$.

\begin{remark}
  As in \cite[\S4.3.8]{michel-2009} or \cite{MR656029}, one can
  define a regularized inner product
  $\langle |\theta|^2, \varphi \rangle_{\reg}$ whenever $\varphi$ admits
  an asymptotic expansion near each cusp in terms of finite
  functions $y^{\beta} \log(y)^{m}$
  with exponent of real part
  $\Re(\beta) \neq 1/2$.
  The regularization takes the form
  $\langle |\theta|^2, \varphi \rangle_{\reg}
  := \langle |\theta|^2, \varphi - \mathcal{E}  \rangle$
  for an auxiliary Eisenstein series $\mathcal{E}$.
  The difference
  $\varphi - \mathcal{E}$
  satisfies the
  growth condition \eqref{eq:assumption-decay},
  so the results of this paper apply
  directly to such regularized inner products.
\end{remark}

\begin{remark}
  It follows in particular from \eqref{eq:theta-expand-pointwise} that
  \begin{equation}\label{cor:theta-orthogonal-cusp-forms}
    \text{$|\theta|^2$ is orthogonal to every cusp form,}
  \end{equation}
  as does not appear to be widely known;
  this feature
  is crucial for the application to subconvexity pursued in \cite{nelson-subconvex-reduction-eisenstein}.
  By taking a residue in
  Shimura's integral \cite{Sh75},
  \eqref{cor:theta-orthogonal-cusp-forms} is equivalent to the well-known fact
  that
  $L(\ad \varphi,s)$ is holomorphic at $s=1$ for
  every cusp form $\varphi$; compare also with \cite{kaplan-yamana-twisted-sym-2}.
  The proof given here is not directly dependent on such considerations.
\end{remark}
\begin{remark}
  In more high-level terms, our arguments amount to
  viewing $|\theta|^2$ as the restriction to the first factor of
  a theta kernel for $(\SL_2, \O_2)$, where $\O_2$ denotes the
  orthogonal group of the split binary quadratic form
  $(x,y) \mapsto x^2 - y^2$; the expansion
  \eqref{eq:theta-expand-pointwise} then amounts to the
  (regularized) decomposition of that kernel with respect to the
  $\O_2$-action.  We may then understand
  \eqref{cor:theta-orthogonal-cusp-forms} as asserting that cusp
  forms do not participate in the global theta correspondence with the
  split $\O_2$, as should be well-known to experts.
\end{remark}

This paper is organized as follows.  In
\S\ref{sec:intro-main-results}, we formulate our main results.
In \S\ref{sec:basic-idea}, we sketch a direct proof of the
simplest special case stated above (Theorem
\ref{thm:hecke-equid-theta}).  The general case is treated in
\S\ref{sec-5} after some local (\S\ref{sec-3}) and global
(\S\ref{sec:global-prelims}) preliminaries.

% The observation that analogues
% of \eqref{eq:plancherel-standard}
% for non-square integrable functions
% might be useful is not new.
% For instance, in the
% periods approach to the subconvexity problem
% on $\GL_2$ following Venkatesh \cite{venkatesh-2005} and Michel--Venkatesh \cite{michel-2009},
% the basic quantitative input is an analogue of \eqref{eq:hecke-equid}
% for $\varphi_1 = |E|^2$ the squared magnitude of a
% unitary Eisenstein series
% and $\varphi_2 = |\Phi|^2$
% that of a cusp form,
% so that $\varphi_2$ is rapidly-decaying
% but $\varphi_1$ fails to belong to $L^2$, or even to $L^1$.
% The standard Plancherel formula does not apply,
% and indeed, its ``formal application'' gives the wrong answer.
% There arises the need for a \emph{regularized Plancherel
% formula}
% which the authors of \cite[\S4.3.8]{michel-2009}
% develop in generality sufficient for their purposes.
% Unfortunately, it does not apply
% to the problem considered here.\footnote{
% In the asymptotic expansion $|\theta|^2(z) \sim y^{1/2} +
% \dotsb$
% near the cusps, the exponent $1/2$ lies on the unitary axis,
% which invalidates the hypotheses of \cite[\S4.3.8]{michel-2009}
% as well as the basic method:
% there is no non-unitary Eisenstein series $\mathcal{E}$ for
% which $|\theta|^2 - \mathcal{E} \in L^2$.  However, see Remark \ref{rmk:other-proofs}.
% }

\subsection*{Acknowledgements}
We gratefully acknowledge the support of
NSF grant OISE-1064866
and SNF grant SNF-137488 during the work leading
to this paper.

\section{Statements of main results}\label{sec:intro-main-results}
% We begin by introducing
% some notation concerning the parametrization
% of elementary theta functions via the Weil representation.
We refer to \S\ref{sec:global-prelims} for detailed definitions
in what follows and to \cite{MR0165033, MR0389772, MR0424695,
  MR577359} for general background.  

Let $k$ be a
global field of characteristic $\neq 2$ with adele ring
$\mathbb{A}$.  Let
$\psi : \mathbb{A}/k \rightarrow \mathbb{C}^{(1)}$ be a
non-trivial additive character.  Denote by $\Mp_2(\mathbb{A})$
the metaplectic double cover of $\SL_2(\mathbb{A})$,
$\mathcal{A}(\SL_2)$ (resp. $\mathcal{A}(\Mp_2)$) the space of
automorphic forms on $\SL_2(k) \backslash \SL_2(\mathbb{A})$
(resp. $\SL_2(k) \backslash \Mp_2(\mathbb{A})$), $\omega$ the
Weil representation of $\Mp_2(\mathbb{A})$ attached to $\psi$
and the dual pair $(\Mp_2, \O(1))$,
$\omega_{\psi} \ni \phi \mapsto \theta_{\psi,\phi} \in \mathcal{A}(\Mp_2)$ the standard theta intertwiner
parametrizing the elementary theta functions,
$\mathcal{I}(\chi)$ the unitary induction to $\SL_2(\mathbb{A})$
of a unitary character
$\chi : \mathbb{A}^\times / k^\times \rightarrow
\mathbb{C}^{(1)}$,
and
$\Eis : \mathcal{I}(\chi) \rightarrow \mathcal{A}(\SL_2)$ the
standard Eisenstein intertwiner defined by averaging over
$\left\{ \left(
    \begin{smallmatrix}
      \ast & \ast \\
      & \ast
    \end{smallmatrix}
  \right) \right\} \backslash
\SL_2(k)$ and analytic continuation along flat sections.
% and
% $\Eis^* : \mathcal{I}(\chi) \rightarrow \mathcal{A}(\SL_2)$
% the normalized Eisenstein intertwiner
% given by $\Eis^* = \Lambda(\chi,1) \Eis$.
Choose a Haar measure $d^\times y$ on $\mathbb{A}^\times$, hence on $\mathbb{A}^\times/k^\times$.

Using $d^\times y$, we define in \S\ref{sec-4-8}
for all nontrivial unitary characters
$\chi$ of $\mathbb{A}^\times /k^\times$ an
$\Mp_2(\mathbb{A})$-equivariant map
$I_\chi : \omega_{\psi} \otimes \overline{\omega_{\psi}}
\rightarrow \mathcal{I}(\chi)$;
it may be regarded as a restricted tensor product of local maps,
regularized by the local factors of $L(\chi,1)$.
As we explain
in \S\ref{sec-4-8},
the composition $\Eis \circ I_\chi$
makes sense for all unitary $\chi$.

Denote by $\int_{(0)}$ the integral over unitary
characters $\chi$  of $\mathbb{A}^\times / k^\times$
with respect to the measure dual to $d^\times y$.
Write simply $\int$
for an integral over $\SL_2(k) \backslash \SL_2(\mathbb{A})$
with respect
to Tamagawa measure, or equivalently, the probability Haar.

The map $\omega_{\psi} \ni \phi \mapsto \theta_{\psi,\phi}$
has kernel given by the subspace
of odd functions,
so we restrict attention
to the even subspace $\omega^{(+)}_\psi = \{\phi \in \omega_\psi
: \phi(x) = \phi(-x)\}$.
(That subspace is reducible, but its further reduction is unimportant for us.)
\begin{theorem}\label{thm:theta-plancherel-redux}
  Let $\phi_1, \phi_2 \in \omega_{\psi}$
  We have the pointwise expansion: for $\sigma \in
  \Mp_2(\mathbb{A})$,
  \begin{equation}\label{eqn:theta-plancherel-redux-pre}
    \theta_{\psi,\phi_1}(\sigma)
    \overline{\theta_{\psi,\phi_2}}(\sigma)
    =
    \int
    \theta_{\psi,\phi_1}
    \overline{\theta_{\psi,\phi_2}}
    +  
    \int_{(0)}
    \Eis(I_{\chi}(\phi_1,\overline{\phi_2}))(\sigma)
  \end{equation}
  Moreover, if $\phi_1, \phi_2 \in \omega^{(+)}_{\psi}$,
  then we have the inner product formula
  \begin{equation}\label{eq:theta-norm}
    \int \theta_{\psi,\phi_1} \overline{\theta_{\psi,\phi_2}}
    =
    2 \int_{\mathbb{A}} \phi_1 \overline{\phi_2}.
  \end{equation}
  Finally, 
  for $\varphi \in \mathcal{A}(\SL_2)$ satisfying the growth condition
  \eqref{eq:growth-estimate-spectral-theorem-theta}
  analogous to \eqref{eq:assumption-decay},
  we have the inner product expansion
  \begin{equation}\label{eqn:theta-plancherel-redux}
    \int
    \theta_{\psi,\phi_1}
    \overline{\theta_{\psi,\phi_2}}
    \varphi 
    =
    \int
    \theta_{\psi,\phi_1}
    \overline{\theta_{\psi,\phi_2}}
    \int 
    \varphi 
    +  
    \int_{(0)}
    \int
    \Eis(I_{\chi}(\phi_1,\overline{\phi_2}))
    \varphi.
  \end{equation}
\end{theorem}
Theorem \ref{thm:theta-plancherel-redux} specializes to
\eqref{eq:theta-expand-pointwise}
upon taking $k = \mathbb{Q}$ and
$\phi_1,\phi_2$ as in the example of \S\ref{sec-4-3}.  A special
case of Theorem \ref{thm:theta-plancherel-redux} was proved by
us in \cite[\S10]{nelson-variance-73-2}; the proofs are presented
differently, and their comparison may be instructive.
The contribution from the trivial character $\chi$
to the RHS of \eqref{eqn:theta-plancherel-redux}
should be compared with the
case of Siegel--Weil discussed in \cite[\S7.2]{2012arXiv1207.4709T}.

More generally, suppose
$\phi_1 \in \omega_{\psi}$,
$\phi_2 \in \omega_{\psi'}$
for some nontrivial characters
$\psi, \psi '$ of $\mathbb{A}/k$.
One can write $\psi '(x) = \psi(a x)$ for some $a \in k^\times$.
If $a \in k^{\times 2}$,
then $\omega_{\psi} \cong  \omega_{\psi '}$,
and so one can study $\int \theta_{\psi, \phi_1}
\overline{\theta_{\psi ',\phi_2}}$
using
Theorem \ref{thm:theta-plancherel-redux}.
If $a \notin k^{\times 2}$,
one can prove (more easily) an analogue
of Theorem \ref{thm:theta-plancherel-redux}
involving dihedral forms for the quadratic space
$k^2 \ni (x,y) \mapsto x^2 - ay^2$; see \S\ref{sec:anisotropic-case}.
One finds in particular that
\begin{equation}\label{eq:orthogonality-distinct-chars}
  \int \theta_{\psi, \phi_1} \overline{\theta_{\psi ', \phi_2}} = 0.
\end{equation}

Following \cite[\S2]{michel-2009},
we employ unitary Sobolev norms $\mathcal{S}_d$ on $\omega$
and automorphic Sobolev norms $\mathcal{S}_d^{\mathbf{X}}$ on
$\mathcal{A}(\SL_2)$
(see \S\ref{sec:local-sobolev-norms}, \S\ref{sec-4-2}).
Denote by $\Xi : \SL_2(\mathbb{A}) \rightarrow \mathbb{R}_{>0}$
the Harish--Chandra spherical function (\S\ref{sec-4-9}).
\begin{theorem}\label{thm:hecke-equid-redux}
  There exists an integer $d$ depending
  only upon the degree of $k$
  so that for
  nontrivial characters $\psi, \psi '$ of $\mathbb{A}/k$
  and
  $\phi_1 \in \omega_{\psi}, \phi_2 \in \omega_{\psi '}$,
  $\varphi \in \mathcal{A}(\SL_2)$
  and $\sigma \in \SL_2(\mathbb{A})$,
  the right translate $\sigma \varphi(x) := \varphi(x \sigma)$
  satisfies
  \begin{equation}\label{eqn:theta-sobolev-error}
    \int \theta_{\psi,\phi_1} \overline{\theta_{\psi',\phi_2}} \cdot \sigma \varphi
    = 
    \int \theta_{\psi,\phi_1} \overline{\theta_{\psi',\phi_2}} \int \varphi
    + O(\Xi(\sigma) \mathcal{S}_d(\phi_1) \mathcal{S}_d(\phi_2) \mathcal{S}_d^{\mathbf{X}}(\varphi)).
  \end{equation}
  The implied constant depends at most upon $(k,\psi,\psi ')$.
\end{theorem}
One should understand the
conclusion of Theorem \ref{thm:hecke-equid-redux} as follows:
if $\theta_1, \theta_2$ are a pair of essentially fixed
elementary theta functions, $\varphi$ is an essentially
fixed automorphic form on $\SL_2$ of sufficient decay,
and $\sigma \in \SL_2(\mathbb{A})$ traverses a sequence
that eventually escapes any fixed compact,
then
$\langle \theta_1 \overline{\theta_2}, \sigma \varphi \rangle$
tends to $\langle \theta_1 \overline{\theta_2}, 1 \rangle
\langle 1, \varphi \rangle$ as
rapidly as the Ramanujan conjecture would predict
if
$\theta_1 \overline{\theta_2}$ were square-integrable,
and with polynomial dependence on the parameters of the
``essentially fixed'' quantities.
An inspection of the proof also reveals polynomial dependence
upon the heights
of the characters $\psi, \psi '$.
The estimate
\eqref{eqn:theta-sobolev-error} can be sharpened a bit
at the cost of lengthening the argument (see e.g. Remark \ref{rmk:sharpenings}),
but already suffices for our intended applications;
the groundwork has been laid here for the pursuit of more specialized refinements
should motivation arise.
% Our
% focus is on foundational aspects of the expansion
% \eqref{eqn:theta-plancherel-redux}.

% Most of the work in this paper is in proving Theorem
% \ref{thm:theta-plancherel-redux}.  The deduction of Theorem
% \ref{thm:hecke-equid-redux} is very short thanks to a
% systematic use of basic properties of adelic Sobolev norms.

\begin{remark}
  We have already indicated
  that Theorem
  \ref{thm:hecke-equid-theta}
  follows by specializing Theorem
  \ref{thm:theta-plancherel-redux}
  to
  \eqref{eq:theta-expand}.
  Alternatively, the $T_{n^2}$ case of Theorem
  \ref{thm:hecke-equid-theta} may be recovered from Theorem
  \ref{thm:hecke-equid-redux} by taking $k = \mathbb{Q}$ and for
  $\sigma$ the finite-adelic matrix $\diag(n,1/n)$.  One can
  deduce from (\ref{eqn:theta-plancherel-redux}) a more general
  form of \eqref{eqn:theta-sobolev-error} involving an extension
  of $\theta_{\phi_1} \overline{\theta_{\phi_2}}$ to the
  similitude group $\PGL_2(\mathbb{A})$ which then specializes to
  the general case of Theorem \ref{thm:hecke-equid-theta};
  it is not clear to us how best to formulate such an extension, and our
  immediate applications do not require it, so we omit it.
\end{remark}

\section{Sketch of proof in the simplest case\label{sec:basic-idea}}
\label{sec-2}
We sketch here the proof of the expansion
\eqref{eq:theta-expand-pointwise}
underlying the proof of Theorem \ref{thm:hecke-equid-theta}.
Let $\theta$ be as in \eqref{eq:jacobi-theta-defn-booyah}.
Set
$z := x + i y$, $e(z) := e^{2 \pi i z}$.
Consider the Fourier expansion
\begin{align*}
|\theta|^2(z)
&=
y^{1/2} \sum_{m,n \in \mathbb{Z}}
  e(m^2 z) \overline{e(n^2 z)}
  \\
  &=
y^{1/2} \sum_{m,n \in \mathbb{Z}}
e((m^2 - n^2) x)
\exp(-2 \pi  (m^2 + n^2) y).
\end{align*}
Change variables to
$\mu := m - n$ and $\nu := m + n$,
so that
$m^2 - n^2 = \mu \nu$
and $m^2 + n^2 = (\mu^2 + \nu^2)/2$:
\[
|\theta|^2(z)
=
y^{1/2}
\sum_{
  \mu,\nu \in \mathbb{Z}
}
e(\mu \nu x)
\exp(- \pi (\mu^2 + \nu^2) y)
1_{\mu \equiv \nu(2)}.
\]
Detect the condition $1_{\mu \equiv \nu(2)}$
as $\frac{1}{2} \sum_{\xi =0,1}
(-1)^{\xi \mu} e^{ \pi i \xi \nu }$
and apply Poisson summation
to (say) the $\nu$ sum:
\begin{align}\label{eq:theta-sum-after-poisson}
  |\theta|^2(z)
  &=
    \frac{1}{2} 
    \sum_{\xi=0,1}
    \sum_{\mu, \nu  \in \mathbb{Z}}
    (-1)^{\xi \mu }
    \exp(- \pi ((\mu x + \tfrac{\xi }{2} + \nu)^2/y + \mu^2 y)) \\
  &= \label{eq:theta-sum-to-be-simplified}
    \frac{1}{2} \sum_{\mu \in \mathbb{Z}, \nu \in \frac{1}{2}
    \mathbb{Z} }
    (-1)^{\mu \nu}
    \exp(- \pi ((\mu x + \nu)^2/y + \mu^2 y)).
\end{align}
% Thus $|\theta|^2$ is a sort of Eisenstein series; Corollary \ref{cor:theta-orthogonal-cusp-forms} follows easily,
% while Theorems \ref{thm:hecke-equid-theta} and \ref{thm:theta-expand} are not far off.
% Since the present section serves a
% motivational rather than logically necessary purpose, we
For simplicity,
we now consider instead of $|\theta|^2(z)$ the
closely related sum
\begin{equation}\label{eq:theta-sum-after-poisson-simplified}
  E(z) := \sum_{\mu,\nu \in \mathbb{Z}} \exp(- \pi ( (\mu x + \nu)^2
  / y + \mu^2 y))
\end{equation}
obtained by stripping \eqref{eq:theta-sum-to-be-simplified}
of its $2$-adic factors.
By isolating the contribution of $(\mu,\nu) = (0,0)$
and
writing the remaining pairs $(\mu,\nu)$
in the form
$(\lambda c,\lambda d)$
for some unique up to sign nonzero integers $\lambda, c, d$
with $\gcd(c,d) = 1$,
so that
% and noting that
$(\mu x + \nu)^2 / y + \mu^2 y
% = ((c y)^2 + (c x + d)^2) \lambda^2/y
= |c z + d|^2 \lambda^2/y$,
we obtain
\[
E(z)
=
1
+
\sum_{
  \substack{
    (c,d) \in (\mathbb{Z}^2 - \{(0,0)\}) / \{\pm 1\} \\
    \gcd(c,d) = 1
  }
}
f\left( \frac{y}{|c z + d|^2} \right)
\]
with
$f(y)
:=
\sum_{\lambda \in \mathbb{Z} - \{0\}}
\exp(- \pi \lambda^2 / y)$.
The function $f$
decays rapidly as $y \rightarrow 0$
and satisfies $f(y) \sim y^{1/2}$ as $y \rightarrow \infty$.
Its Mellin transform $\widetilde{f}(s) := \int_{y \in
  \mathbb{R}_+^\times}
f(y) y^{-s} \, d^\times y$
is given for $\Re(s) > 1/2$ by
\[\widetilde{f}(s)
=
\int_{y \in \mathbb{R}_+^\times}
\sum_{\lambda \in \mathbb{Z} - \{0\}}
\exp(- \pi \lambda^2 / y)
y^{-s} \, d^\times y
=
2 \xi(2 s).\]
By Mellin inversion,
$f(y)
=
\int_{(2)}
2 \xi(2 s)
y^s \, \frac{d s}{2 \pi i}$.
Thus
$E(z)
= 1
+
\int_{(2)}
E^*_s(z)
\, \frac{d s}{2 \pi i }$,
where
$E^*_s
:=
2 \xi(2 s)
E_s$
with $E_s$ as in \eqref{eq:defn-classical-eis-series}.
It is known that $E_s$ vanishes to order one as $s \rightarrow 1/2$,
while $E^*_s$ is holomorphic for $\Re(s) \geq 1/2$
except for a simple pole at $s=1$ of constant residue $1$.
Shifting contours, we obtain
\begin{equation}\label{eq:toy-decomp}
  E(z)
  = 2 + \int_{(1/2)} E^*_s(z) \, \frac{d s}{2 \pi i }.
\end{equation}
By standard bounds
on $E_s^*(z)$ that take into account the rapid decay of the factor $\Gamma(s)$,
\[
\int_{(1/2)}
\int_{\Gamma \backslash \mathbb{H}} \htt(z)^{1/2 - \delta} |E_s^*(z)| < \infty.
\]
Therefore \eqref{eq:toy-decomp}
holds not only pointwise
but also weakly when tested against
functions $\varphi$ satisfying \eqref{eq:assumption-decay}.
In particular,
$\langle E, 1 \rangle = \langle 2,1 \rangle = 2$.

The expansion \eqref{eq:theta-expand-pointwise}
follows from the above argument
applied to $|\theta|^2(z)$ rather than $E(z)$,
or alternatively, by specializing Theorem
\ref{thm:theta-plancherel-redux};
see also \cite[\S9]{nelson-variance-73-2}.

In passing to the general results of
\S\ref{sec:intro-main-results}, we must keep track of how more
complicated variants of the $2$-adic factors in
\eqref{eq:theta-sum-to-be-simplified} affect the residue
arising in the contour shift.  This is ultimately achieved by
the inversion formula for an adelic partial Fourier transform.
We must also quantify everything; we do so crudely.
\begin{remark}\label{rmk:other-proofs}
  The proof sketched above and its generalization given below
  is the third that we have found.
  The following alternative arguments are possible,
  but less efficient:
  \begin{enumerate}
  \item One can realize $\theta$ as the residue as
    $\eps \rightarrow 0$ of a $1/2$-integral weight Eisenstein
    series $\tilde{E}_{3/4+\eps}$ and then subtract off a weight
    zero Eisenstein series $\mathcal{E}_{1+\eps}$ to regularize
    the inner product
    $\langle \overline{\theta } \tilde{E}_{3/4+\eps}, \varphi
    \rangle_{\reg} := \langle \overline{\theta }
    \tilde{E}_{3/4+\eps} - \mathcal{E}_{1+\eps}, \varphi
    \rangle$
    following the scheme of \cite[\S 4.3.5]{michel-2009} with
    some necessary modifications; the hypotheses do not
    literally apply, but the method can be adapted with some
    work.  One can then extract the residue of the regularized
    spectral expansion of this regularized inner product to
    obtain the required formula for
    $\langle |\theta|^2, \varphi \rangle$.
  \item As in the proof of the standard Plancherel formula (see e.g. \cite{MR1942691, MR546600}), one can reduce first to
    understanding $\langle |\theta|^2, \varphi \rangle$ when
    $\varphi$ is an incomplete Eisenstein series, expand
    $\varphi = \int_{(2)} c(s) E_s \, \frac{d s}{2 \pi i }$ as
    an integral of spectral Eisenstein series $E_s$, and then shift contours
    to the line $\Re(s)=1/2$.  Difficulty arises (especially in the generality of Theorem \ref{thm:theta-plancherel-redux})
    when one wishes to compare the expansion so-obtained
    to the spectral coefficients
    $\langle E_{1/2+i t}, \varphi \rangle$ of $\varphi$, which
    are given by $c(1/2+i t) + M(1/2 + it) c(1/2 - it)$ rather
    than $c(1/2+i t)$
    (here
    $M(s) := \xi(2 s) / \xi(2(1-s))$).
    The analogous difficulty in the proof of the standard
    Plancherel formula is addressed by the functional equation
    for the intertwining operators, of which some more complicated
    variants are required here.
    The present approach is more direct.
  \end{enumerate}
\end{remark}

\section{Local preliminaries}
\label{sec-3}
\subsection{Generalities\label{sec:local-generalities}}
\label{sec-3-1}
In this section we work over a local field $k$ of characteristic
$\neq 2$.
Let $\psi : k \rightarrow \mathbb{C}^{(1)}$ be a nontrivial character.
Equip $k$ with the Haar measure self-dual for the character $\psi_2$ defined by $\psi_2(x) := \psi(2 x)$.

When $k$ is non-archimedean, we denote by $\mathfrak{o}$ its
ring of integers,
$\mathfrak{p}$ its maximal ideal, and $q := \#
\mathfrak{o}/\mathfrak{p}$ the
cardinality of its residue field.

\subsection{``The unramified case''\label{sec:unram-case}}
\label{sec-3-2}
We use this phrase
to mean specifically
that $k$ is non-archimedean,  $\psi$ is unramified, and the residue characteristic of $k$ is $\neq 2$.
\subsection{Conventions on multiplicative characters\label{sec:conventions-characters}}
\label{sec-3-3}
We represent the group
\[\mathfrak{X}(k^\times) := \Hom(k^\times, \mathbb{C}^\times)\] of
continuous homomorphisms \emph{additively}.  Denote
by $y^{\chi}$ the value taken by  the character
$\chi \in \mathfrak{X}(k^\times)$ at the element $y \in k^\times$,
by $0
\in \mathfrak{X}(k^\times)$ the trivial character $y \mapsto y^0
:= 1$, by
$\alpha$
the absolute value character $y
\mapsto y^\alpha := |y|$, and, for any complex number $s$, by $s
\alpha$ the character $y \mapsto y^{s \alpha} :=
|y|^s$; thus
\[
y^{\chi_1 + \chi_2} = y^{\chi_1} y^{\chi_2},
\quad 
(y_1 y_2)^{\chi}
= y_1^{\chi} y_2^{\chi},
\quad y^{-\chi} = (1/y)^{\chi},
\quad
y^0 = 1.
\]
Every $\chi$ may be written uniquely
as $c \alpha + \chi_0$ for some $c \in \mathbb{R}$
and $\chi_0$ unitary;
$\Re(\chi) := c$
is called the \emph{real part} of $\chi$.
\subsection{The metaplectic group\label{sec:metaplectic-local}}
\label{sec-3-4}
Denote by $\Mp_2(k)$ the metaplectic double cover of
$\SL_2(k)$,
defined using Kubota cocycles \cite{MR0204422}
as the set of all pairs
$(\sigma,\zeta) \in \SL_2(k) \times \{\pm 1\}$
with the multiplication law
$(\sigma_1,\zeta_1) (\sigma_2, \zeta_2)
= (\sigma_1 \sigma_2, \zeta_1 \zeta_2 c(\sigma_1,\sigma_2))$
where
\begin{equation}\label{eq:local-kubota-2}
  c(\sigma_1,\sigma_2)
  := \left(
    \frac{x(\sigma_1 \sigma_2)}{ x(\sigma_1)},
    \frac{x(\sigma_1 \sigma_2)}{ x(\sigma_2)}
  \right),
  \quad x
  \left(
    \begin{pmatrix}
      \ast & \ast \\
      c & d
    \end{pmatrix}
  \right)
  := \begin{cases}
    d & \text{ if }c = 0, \\
    c & \text{ if } c \neq 0
  \end{cases}
\end{equation}
with $(,) : k^\times/k^{\times ^2}  \times
k^\times/k^{\times ^2} 
\rightarrow \{\pm 1\}$ the Hilbert symbol.
As generators for
$\Mp_2(k)$
we take
for $a \in k^\times, b \in k$
and $\zeta \in \{\pm 1\}$
the elements
\[
n(b)
= \left( \begin{pmatrix}
    1 & b \\
    & 1
  \end{pmatrix}, 1 \right),
\quad 
t(a)
= \left( \begin{pmatrix}
    a &  \\
    & a^{-1}
  \end{pmatrix}, 1 \right),
\]
\[
w =
\left( \begin{pmatrix}
    & 1 \\
    -1 & 
  \end{pmatrix}, 1 \right),
\quad
\eps(\zeta) = (1,\zeta)
\]
which
satisfy
the relations
$n(b_1) n(b_2) = n(b_1 + b_2)$,
$t(a_1) t(a_2) = t(a_1 a_2) \eps((a_1,a_2))$,
$t(a) n(b) = n(a^2 b) t(a)$,
$w t(a) = t(a^{-1}) w$,
$w^2 = t(-1)$
and
(when $b \neq 0$)
$w n(-b^{-1}) = n(b) t(b) w n(b) w^{-1}$.
Identify functions on $\SL_2(k)$ with their pullbacks to $\Mp_2(k)$.
\subsection{Principal congruence subgroups}
\label{sec-3-5}
Suppose for this subsection that $k$ is non-archimedean.
Set $K_{\SL_2} := \SL_2(\mathfrak{o})$.
For $m \in \mathbb{Z}_{\geq 0}$,
denote by $K_{\SL_2}[m] := K \cap (1 + \mathfrak{p}^m
M_2(\mathfrak{o}))$ the $m$th
principal congruence subgroup.
Define a map
\[
\sigma : K_{\SL_2} \rightarrow \Mp_2(k)
\]
\[
\sigma \left( \begin{pmatrix}
    \ast  &  \ast  \\
    c & d
  \end{pmatrix} \right)
:=
\begin{cases}
  (c,d)^{\nu(c)} & \text{ if } c \neq 0, \\
  1 & \text{ if } c = 0
\end{cases}
\]
where $\nu$ denotes the normalized valuation on $k$.
There exists an $m_0$, depending only upon $\nu(2)$,
so that the restriction of $\sigma$ to $K_{\SL_2}[m_0]$ is a
homomorphism (\cite[Prop 2.8]{MR0424695}, \cite{MR0255490});
in fact, one may take $m_0 := 0$
in the unramified case (\S\ref{sec:unram-case}).
In general, denote for $m \geq m_0$ by $K[m] :=
\sigma(K_{\SL_2}[m]) < \Mp_2(k)$ the image.
It defines a filtration $K[m_0] \supset K[m_0+1] \supset \dotsb$
of congruence
subgroups of $\Mp_2(k)$.
One has $w \in K[0]$
whenever
$K[0]$ is defined
and $n(b),t(a) \in K[m]$ for all
$b,a \in \mathfrak{p}^m, 1 + \mathfrak{p}^m$
whenever $K[m]$ is defined.

For a smooth representation $V$ of $\Mp_2(k)$
and $m \geq m_0$,
denote by $V[m] := V^{K[m]}$ the subspace of $K[m]$-invariant vectors.
For $m < m_0$, write $V[m] := \{ 0 \}$.
Thus $\{0\} = V[-1] \subseteq V[0] \subseteq V[1] \subseteq \dotsb$ and $V = \cup V[m]$.

\subsection{Sobolev norms\label{sec:local-sobolev-norms}}
\label{sec-3-6}
For each integer $d$ and unitary admissible representation $V$
of $\Mp_2(k)$, denote by $\mathcal{S}_d^V$ the Sobolev norm on
$V$ defined by the recipe of \cite[\S2]{michel-2009}.
Strictly speaking, that article considers the case of
reductive groups and not their finite covers, but the
definitions and results apply verbatim in our context (using the
principal subgroups defined above in the non-archimedean
case).

These norms have the shape
$\mathcal{S}_d^V(v) := \|\Delta^d v\|$ for a positive
self-adjoint operator $\Delta$ on $V$, whose definition is given
the archimedean case by (essentially)
$\Delta := 1 - \sum_{X \in \mathcal{B}(\Lie \Mp_2(k))} X^2$ and
in the non-archimedean case by multiplication by $q^m$ on the
orthogonal complement in $V[m]$ of $V[m-1]$.  A number of useful
properties of such norms (``axioms (S1a) through (S4d)'') are
established in \cite[\S2]{michel-2009}; for the purposes of
\S\ref{sec-3}, we shall need only the following:
\begin{enumerate}
\item [(S1b)]
  \emph{The distortion property.}
  There is a constant $\kappa$,
  depending only upon $\deg(k)$,
  so that for all $g,v \in \Mp_2(k), V$, one has
  $\mathcal{S}_d^V(g v)
  \ll \|\Ad(g)\|^{\kappa d} \mathcal{S}_d^V(v)$
\item [(S4d)]
  \emph{Reduction to the case of $\Delta$-eigenfunctions.}
  If $W$ is a normed vector space and $\ell : V \rightarrow W$ a
  linear functional
  with the property that
  $|\ell(v)| \leq A \|\Delta^d v\|_V$ for each
  $\Delta$-eigenfunction $v$,
  then
  $|\ell(v)| \leq A' \mathcal{S}_{d'}^V(v)$ for all $v \in V$,
  where $(A',d')$ depends only upon $(A,d)$
  and $\deg(k)$.
\end{enumerate}
When the representation $V$ is clear from context,
we abbreviate $\mathcal{S}_d := \mathcal{S}_d^V$.

\subsection{Conventions on implied constants}
\label{sec-3-7}
Implied constants in this section are allowed
to depend upon $(k,\psi)$ except in the unramified case (\S\ref{sec:unram-case}),
in which implied constants are required to be depend at most upon $\deg(k)$.
Similarly, we abbreviate $\mathcal{S} := \mathcal{S}_d$ when $d$ (the \emph{implied index}) may be chosen
with the above dependencies.
The purpose of this convention is to ensure that
implied constants are uniform
when $(k,\psi)$ traverses the local components of analogous global data.

\subsection{The Weil representation\label{sec:weil-rep-local}}
\label{sec-3-8}
For $(V,q)$ a quadratic space over $k$,
the Weil representation $\omega_{\psi,V}$
of $\Mp_2(k)$ on the Schwartz--Bruhat space $\mathcal{S}(V)$
is defined on the generators as follows:
there is a
quartic character
$\chi_{\psi,V}: k^\times \rightarrow \mu_4 < \mathbb{C}^{(1)}$
and an eighth root of unity
$\gamma_{\psi,V} \in  \mu_8 < \mathbb{C}^{(1)}$,
whose precise definitions are unimportant for our purposes (see \cite{MR0165033, MR0401654, MR0424695} for details),
so that
\[
\omega_{\psi,V}(n(b))
\phi(x)
=
\psi  (b q(x))
\phi(x),
\quad 
\omega_{\psi,V}(t(a))
\phi(x)
= a^{\chi_{\psi,V} + (d/2) \alpha}
\phi(a x),
\]
\[
\omega_{\psi,V}(w)
\phi
=
\gamma_{\psi,V}
\phi^\wedge,
\quad 
\omega_{\psi,V}(\eps(\zeta))
\phi
=
\zeta^{\dim V} \phi
\]
where
$\phi^\wedge(y) :=
\int _{V} \phi(x) \psi (\langle x, y
\rangle) \, d x$
with $\langle x,y \rangle := q(x+y)-q(x)-q(y)$
and the Haar measure $d x$ normalized so that
$((\phi)^\wedge)^\wedge(x) = \phi(-x)$.
The assignment $V \mapsto \omega_{\psi,V}$ is compatible with direct sums
in the sense that if $V = V' \oplus V''$, then
$\omega_{\psi,V} = \omega_{\psi,V'} \otimes \omega_{\psi,V''}$
with respect to the dense inclusion
$\mathcal{S}(V') \otimes \mathcal{S}(V'') \hookrightarrow
\mathcal{S}(V)$.
The complex conjugate representation is given by $\overline{\omega_{\psi,V}} \cong
\omega_{\psi,V^-}$ where $V^- := (V,-q)$ denotes the quadratic space ``opposite'' to $V = (V,q)$ obtained by negating the quadratic form.

We are concerned here primarily, although not exclusively,
with the case that $V$ is the
one-dimensional quadratic space $V_1 \cong k$ with the quadratic form
$x \mapsto x^2$.
In that case, we write simply $\omega_{\psi} :=
\omega_{V_1,\psi}$.
Since $\psi$ is fixed throughout \S\ref{sec-3},
we accordingly abbreviate
$\omega := \omega_{\psi}$.
It is realized on the space $\mathcal{S}(k)$.  Note that the
Fourier transform $\phi \mapsto \phi^\wedge$ attached above to
this space differs from the ``usual one'' by a factor of $2$ in
the argument of the phase, i.e.,
$\phi^\wedge(x) = \int_{y \in V_1} \phi(y) \psi(2 x y) \, d y$.
We normalized the Haar measure on $k$
as we did
in \S\ref{sec:local-generalities}
so that
the isomorphism $V_1 \cong k$ is measure-preserving and $\omega$ is unitary.
In the unramified case,
the space $\omega^{K[0]}$ of $K[0]$-invariant vectors in $\omega$ is
one-dimensional
and spanned by the characteristic function $1_\mathfrak{o}$ of the maximal order $\mathfrak{o}$ in $k$.

\subsection{Basic estimates in the Weil representation}
\label{sec-3-9}
The following estimate is cheap, but adequate for us.
\begin{lemma}\label{lem:local-bound-1}
  For $\phi  \in \omega$,
  one has $\|\phi\|_{L^1} \ll \mathcal{S}(\phi)$
  and
  $\|\phi\|_{L^\infty} \ll \mathcal{S}(\phi)$.
\end{lemma}
\begin{proof}
  The $L^\infty$ bound follows from the $L^1$-bound, Fourier
  inversion and the distortion property
  applied to the Weyl element $w$.
  We turn to the $L^1$-bound.
  In the real case $k = \mathbb{R}$, there is an element $Z$ in the complexified
  Lie algebra of $\Mp_2(k)$ so that
  $\omega(Z) \phi(x) = x^2 \phi(x)$.
  By Cauchy--Schwarz,
  the contribution to $\|\phi\|_{L^1}$ from the range $|x| \leq 1$
  is bounded by $O(\|\phi\|)$
  and that from the remaining range
  by
  $\int_{x \in k : |x| > 1} |\phi(x)|
  =  \int_{x \in k : |x| > 1} |x|^{-2} |\omega(Z) \phi(x)|
  \ll \|\omega(Z) \phi\| \ll \mathcal{S}_1(\phi)$.
  The complex case is similar.
  In the non-archimedean case,
  it suffices by reduction to the case of $\Delta$-eigenfunctions (S4d)
  to show for each
  $m \geq 0$ and $\phi \in \omega[m]$
  that $\|\phi\|_{L^1} \ll q^{A m} \|\phi\|$
  for some absolute $A$.
  The condition $\phi \in \omega[m]$ implies that
  for all $b \in
  \mathfrak{p}^m$, one has
  $\omega(n(b)) \phi = \phi$,
  that is to say,
  $(\psi(b x^2) - 1) \phi(x) = 0$ for all $x \in k$.
  Therefore $\phi$ is supported on elements $x \in k$ satisfying
  the constraint
  $\psi(b x^2) = 1$ for all $b \in \mathfrak{p}^m$.
  That constraint implies $|x| \ll q^{m/2}$, and the set of elements satisfying it has
  volume $O(q^{m/2})$,
  so Cauchy--Schwarz gives as required that $\|\phi\|_{L^1} \ll q^{m/4}
  \|\phi\| \ll q^{O(m)} \|\phi\|$.
\end{proof}

\subsection{Induced representations\label{sec:local-induced-reps}}
\label{sec-3-10}
Denote by $\mathcal{I}(\chi)$ the unitarily normalized
induction
to $\SL_2(k)$ of a character $\chi$ of $k^\times$,
realized in its induced model as a space of functions
$f : \SL_2(k) \rightarrow \mathbb{C}$
satisfying
$f(n(x) t(y) g) = y^{\alpha+\chi} f(g)$
for all $x,y,g \in k,k^\times,\SL_2(k)$.
When $\chi$ is unitary, $\mathcal{I}(\chi)$ is unitary
with respect to the norm
$\|f\| := (\int_{K_{\SL_2}} |f|^2)^{1/2}$
for $K_{\SL_2} \leq \SL_2(k)$ the standard maximal compact subgroup
equipped with the probability Haar.
When $\chi$ is unitary, $\mathcal{I}(\chi)$ is irreducible
if and only if $\chi$ is not a non-trivial quadratic
character.

\subsection{Change of polarization}
\label{sec-3-11}
Recall from \S\ref{sec:weil-rep-local} that $V_1 \cong k$ is the one-dimensional quadratic space
with the form $x \mapsto x^2$ underlying $\omega$ and $V_1^-
\cong k$ that with $x \mapsto -x^2$ underlying
$\overline{\omega}$.
We abbreviate the tensor product of $\omega$ and
$\overline{\omega}$
as $\omega^2 := \omega \otimes \overline{\omega }$; it is not in
any literal sense the square of the representation $\omega$,
but we shall have no occasion to refer to the latter.
Then $\omega^2$
identifies with the Weil representation $\omega_{\psi,V_2}$
attached to the quadratic space
$V_2 := V_1 \oplus V_1^- \cong k^2$ equipped with the form
$(x_1,x_2) \mapsto x_1^2 - x_2^2$.  The latter quadratic space
is split, and so by a well-known procedure (see e.g. \cite[\S0, (VII)]{MR743016}) involving a
change of polarization in the symplectic space $W \otimes V_2$
underlying the construction of $\omega_{\psi,V_2}$
(here $W$ is the symplectic space for
which $\SL_2(k) = \Sp(W)$),
there exists an intertwiner
$\mathcal{F} : \mathcal{S}(V_2) \rightarrow \mathcal{S}(W) \cong
\mathcal{S}(k^2)$
under which the representation $\omega^2$ on the source
corresponds to a natural geometric action of
$\Mp_2(k) = \Mp_2(W)$ on the target:

Denote by $V_s \cong k^2$  the standard split quadratic space with the form $(y_1,y_2) \mapsto y_1 y_2$.
The map $\rho : V_s \cong k^2 \rightarrow V_2 \cong k^2$ given by $\rho(y_1,y_2) := \left( \frac{y_1 + y_2}{2}, \frac{y_1 -
    y_2}{2} \right)$ is an isometry:
if $(x_1,x_2) = \rho(y_1,y_2)$, then $x_1^2 - x_2^2 = y_1 y_2$.
Define for $\phi \in \omega \otimes
\overline{\omega}$ the partial Fourier transform
\[
\mathcal{F} \phi(y_1,y_2)
:=
\int_{t \in k}
\phi(\rho(y_1,t))
\psi(y_2 t) \, d t
=
\int_{t \in k}
\phi(\frac{y_1 + t}{2}, \frac{y_1 - t}{2})
\psi(y_2 t) \, d t.
\]
\begin{lemma} \label{lem:partial-fourier-unram}
 In the unramified case, $\mathcal{F} 1_{\mathfrak{o}^2} =  1_{\mathfrak{o}^2}$.
\end{lemma} \begin{proof} By direct calculation. \end{proof}
\begin{lemma}\label{lem:local-change-polarization}
  For $\sigma \in \Mp_2(k)$ and $\phi \in \omega \otimes
  \overline{\omega}$,
  one has
  $\mathcal{F} \omega^2(\sigma) \phi(y)
  = \mathcal{F} \phi(y \sigma)$.
  Here $y \sigma$ denotes the right multiplication of $y$ by the image of $\sigma$ in $\SL_2(k)$.
\end{lemma}
\begin{proof}
  See for instance Jacquet--Langlands \cite[Prop
  1.6]{MR0401654}
  or Bump \cite[Prop 4.8.7]{MR1431508}.
  % .\footnote{ The notation here is
  %   a bit different, so we record the
  % short calculation.  The claim is true
  % when
  % $\sigma = \eps(-1)$ because the latter acts trivially on both sides.
  % It remains to consider cases in which
  % $\sigma = n(b)$ with
  % $b \in k$ or $\sigma = w$ of $\Mp_2(k)$; together with  $\eps(-1)$, such elements generate $\Mp_2(k)$.
  % For $q$ the quadratic form $q(x) = x_1^2 - x_2^2$,
  % recall that $q(\rho(y_1,y_2)) = y_1 y_2$.
  % Then $\omega^2(n(b)) \phi(\rho(y_1,t)) = \psi(b y_1 t) \phi(\rho(y_1,t))$ and so
  % $\mathcal{F} \omega^2 (n(b)) \phi(y) = \int_{t \in k}
  % \phi(\rho(y_1,t)) \psi((y_2 + b y_1)t)
  % = \mathcal{F} \phi(y_1, y_2 + b y_1) = \mathcal{F} \phi(y
  % n(b))$.
  % Similarly,
  % $\omega^2 \phi(\rho(y_1,t))
  % = \int_{z_1,z_2 \in k}
  % \phi(z_1,z_2) \psi( 2 ( \frac{y_1 + y_2}{2} z_1 - \frac{y_1 - y_2}{2} z_2))
  % = \int_{z_1,z_2 \in k}
  % \phi(z_1,z_2) \psi( y_1 (z_1 - z_2) + t (z_1 + z_2))$.
  % Since $\int_{z_2, t \in k}
  % \phi(z_1,z_2) \psi(y_1 (z_1 - z_2) + t (z_1 + z_2 + y_2))
  % = |2| \phi(z_1, - z_1 - y_2) \psi(y_1(2 z_1 + y_2))$
  % and
  % $\phi(z_1, - z_1 - y_2) = \phi(\rho(-y_2,2 z_1 + y_2))$,
  % we obtain
  % $\mathcal{F} \omega^2(w) \phi(y) =|2| \int_{z_1 \in k}
  % \phi(\rho(-y_2, 2 z_1 + y_2)) \psi(y_1 (2 z_1 + y_2))
  % = \int_{z_1 \in k} \phi(\rho(-y_2, t)) \psi(y_1 t) =
  % \mathcal{F} \phi(y w)$, as required.  }
\end{proof}

\subsection{The local intertwiner}
\label{sec-3-12}
Let $\chi$ be a character of $k^\times$
with $\Re(\chi) > -1$.
By Lemma \ref{lem:local-change-polarization}, the map
$I_\chi : \omega \otimes \overline{\omega} \rightarrow
\mathcal{I}(\chi)$
defined
by the convergent integral
\[
I_\chi(\phi)(\sigma)
=
\int_{y \in k^\times}
y^{\alpha + \chi}
\mathcal{F} \phi(y e_2 \sigma)
\, d^\times y
\]
is equivariant.
The normalized local Tate integral
$\chi \mapsto I_\chi(\phi) / L(\chi,1)$
extends to an entire function of $\chi$.
We will ultimately only need to consider the range $\Re(\chi) \geq 0$.
\begin{lemma}\label{lem:local-intertwiner-unramified-case}
  Suppose we are in the unramified case
  (\S\ref{sec:unram-case}).  
  Let $\chi$ be an unramified character of $k^\times$
  with $\chi \neq -\alpha$.
  Then
  $I_\chi(1_{\mathfrak{o}} \otimes 1_{\mathfrak{o}})$ is the $K$-invariant vector
  taking the value $L(\chi,1)$ at the identity.
\end{lemma}
\begin{proof}
  Our assumptions
  imply by Lemma \ref{lem:partial-fourier-unram} that $\mathcal{F} \phi = 1_{\mathfrak{o}^2}$,
  so we conclude by the standard evaluation of unramified local Tate integrals.
\end{proof}
\begin{proposition}\label{prop:local-sobolev-estimate}
  Suppose $\chi$ is unitary, so that $\mathcal{I}(\chi)$ is unitary.
  For each $d$ there exists $d'$ so that for all $\phi = \phi_1 \otimes \overline{\phi_2} \in \omega^2$,
  \[
  \mathcal{S}_d^{\mathcal{I}(\chi)}(I_\chi(\phi))
  \ll \mathcal{S}_{d'}^{\omega}(\phi_1)
  \mathcal{S}_{d'}^{\omega}(\phi_2).
  \]
\end{proposition}
\begin{proof}
  By the equivariance of $I_\chi$,
  we reduce to showing that
  $\|I_\chi(\phi)\| \ll \mathcal{S}(\phi_1) \mathcal{S}(\phi_2)$.
  The norm on $I_\chi(\phi)$ is given by integration over the
  maximal compact,
  so by the distortion property (S1b) and -- once again -- the equivariance of $I_\chi$,
  we reduce to establishing the pointwise bound
  $I_\chi(\phi)(1) \ll \mathcal{S}(\phi)$.
  But
  \[
  I_\chi(\phi)(1)
  \ll
  \int_{y \in k^\times}
  y^{\alpha + \chi}
  \mathcal{F} \phi(y e_2)
  \, d^\times y
  \ll
  \int_{x \in k}
  |\mathcal{F} \phi(0,x) |
  \, d x
  \ll
  \|\phi_1\|_{L^1} \|\phi_2\|_{L^1},
  \]
  so we conclude by Lemma \ref{lem:local-bound-1}.
\end{proof}

\section{Global preliminaries\label{sec:global-prelims}}
\label{sec-4}
\subsection{Fields, groups, spaces, measures}
\label{sec-4-1}
Let $k,\mathbb{A},\psi$ be as in \S\ref{sec:intro-main-results}.  In what
follows, equip all discrete spaces with discrete measures and
quotient spaces with quotient measures.  Equip $\mathbb{A}$ with
Tamagawa measure, so that $\vol(\mathbb{A}/k) = 1$.  Fix an
arbitrary Haar measure $d^\times y$ on $\mathbb{A}^\times$.
Denote by $\mathfrak{p}$ a typical place of
$k$.
Then for all but finitely many $\mathfrak{p}$, the pair
$(k_\mathfrak{p},\psi_\mathfrak{p})$
will be in the ``unramified case'' (\S\ref{sec:unram-case}).

Denote by $\Mp_2(\mathbb{A})$ the metaplectic double cover of
$\SL_2(\mathbb{A})$;
it is the set of pairs
$(\sigma,\zeta) \in \SL_2(\mathbb{A}) \times \{\pm 1\}$
with respect to the multiplication law
$(\sigma_1,\zeta_1) (\sigma_2, \zeta_2)
= (\sigma_1 \sigma_2, \zeta_1 \zeta_2 c(\sigma_1,\sigma_2))$
where $c(\sigma_1,\sigma_2) := \prod_{\mathfrak{p}} c_{\mathfrak{p}}(\sigma_{1,\mathfrak{p}}, \sigma_{2,\mathfrak{p}})$
is the product of local Kubota cocycles
$c_\mathfrak{p}$
defined in \ref{sec:metaplectic-local}.
Identify functions on $\SL_2(\mathbb{A})$
with their (``non-genuine'')
pullbacks to the double cover $\Mp_2(\mathbb{A})$.
Define $n(b), t(a) \in \Mp_2(\mathbb{A})$ for $b,a \in \mathbb{A}, \mathbb{A}^\times$ as in \S\ref{sec:metaplectic-local}.

Denote by $P < \SL_2$ the
upper-triangular subgroup and $U < P$ the strictly
upper-triangular subgroup.  Write $e_1 := (1,0), e_2 := (0,1)$.  Equip
$U(\mathbb{A}), \SL_2(\mathbb{A})$ with Tamagawa measures, so
that the map
$U(\mathbb{A}) \backslash \SL_2(\mathbb{A}) \ni \sigma \mapsto
e_2 \sigma \in \mathbb{A}^2$
is measure-preserving.  Equip $P(\mathbb{A})$ with the left Haar
measure compatible with the natural isomorphism
$P(\mathbb{A})/U(\mathbb{A}) \cong \mathbb{A}^\times$ and the
chosen Haar measure on $\mathbb{A}^\times$.
Set $\mathbf{X} := \SL_2(k) \backslash \SL_2(\mathbb{A})$ equipped
with the Tamagawa measure (i.e., the probability Haar).
We retain and adapt to the adelic setting the conventions of \S\ref{sec:conventions-characters} concerning multiplicative characters.  For instance, we denote now by $y^{\alpha} := |y|$ the adelic absolute value of $y \in \mathbb{A}^\times$.

\subsection{Siegel domains}\label{sec:siegel-domains}
For convergence issues, we assume basic familiarity with Siegel
domains (see e.g. \cite[\S4]{MR0291087} or \cite{MR0244260, MR546600} or \cite[\S12]{MR0244260}); the
reader may alternatively trust that the general analytic issues
concerning convergence
are not qualitatively different from those in the model example of
\S\ref{sec:basic-idea}.
In particular, denote
by $\htt : \mathbf{X} \rightarrow \mathbb{R}_{>0}$ the function
$\htt(g) := \max_{\gamma \in \SL_2(k)} \htt_{\mathbb{A}}(\gamma
g)$
where $\htt_{\mathbb{A}}$ is defined with respect to the Iwasawa
decomposition by $\htt_{\mathbb{A}}(n(x) t(y) k) := |y|^{1/2}$.
Then $\htt(x) \geq c > 0$ for some $c > 0$ depending only upon
$k$; moreover, $\htt$ is proper.

\subsection{Sobolev norms}
\label{sec-4-2}
We briefly recall the adelic Sobolev norms introduced in
\cite[\S2]{michel-2009}
which were inspired in turn by \cite{venkatesh-2005, MR1930758}.
For an
integer $d$ and a unitary admissible representation $V$ of $\Mp_2(\mathbb{A})$, define
the Sobolev norm $\mathcal{S}_d^V$ on $V$ by the formula
$\mathcal{S}_d^V(v) := \|\Delta^d v\|$, where
$\Delta$ denotes the restricted tensor product of
the operators defined in \S\ref{sec:local-sobolev-norms}.  This definition applies also to $\SL_2(\mathbb{A})$-modules, which we regard as non-genuine $\Mp_2(\mathbb{A})$-modules.
These norms take finite values
on smooth vectors and apply in particular when
$V = L^2(\mathbf{X})$, but for that space, a finer Sobolev norm
$\mathcal{S}^\mathbf{X}_d$ is also useful: for $f \in C^\infty(\mathbf{X})$, 
set
$\mathcal{S}_d^\mathbf{X}(f) := \|\htt^d \Delta^d f
\|_{L^2(\mathbf{X})}$.
We omit the superscript, writing $\mathcal{S}_d^V :=
\mathcal{S}_d$,
when $V$ is clear from context.
The indices $d, d'$ appearing here and below are implicitly restricted to
depend only upon $\deg(k)$.  We employ as in \S\ref{sec:local-sobolev-norms}
and \cite[\S2]{michel-2009} the convention of omitting the index
when it is implied.
Note that $\mathcal{S}_d, \mathcal{S}_{-d}$ are dual;
indeed, for $u,v \in V$,
\begin{equation}\label{eq:sobolev-duality}
  \langle u, v \rangle = \langle \Delta^d u, \Delta^{-d} v \rangle.
\end{equation}
As in \cite[\S2.6.5]{michel-2009},
we set
\begin{equation}\label{eq:sobolev-conductor-defn}
  C_{\Sob}(V) := \inf_{v \in V, \|v\| = 1} \|\Delta v\|.
\end{equation}
We now record some specialized and annotated forms of the axioms
from \cite[\S2]{michel-2009}
relevant for \S\ref{sec-4}:
\begin{enumerate}
\item[(S1c)] \emph{Sobolev embedding.}
  Let $V$ be a unitary irreducible admissible representation of $\Mp_2(\mathbb{A})$.
  Then for each $d$ there exists a $d'$
  so that the inclusion of Hilbert spaces
  $(V,\mathcal{S}_{d'}^V) \rightarrow (V,\mathcal{S}_{d}^V)$ is
  trace class;
  moreover (see \cite[\S2.6.3]{michel-2009}),
  there exists $d_0$ so that the trace of $\Delta^{-d_0}$
  is bounded uniformly in $V$
  (with
  the global field
  $k$ held fixed).
  Concretely,
  this gives
  implications of the shape
  \begin{equation}\label{eqn:s1c-super-basic}
    \langle u,v \rangle \ll \mathcal{S}_{-d}(u) C(d')
    \implies \mathcal{S}_d(v) \ll C(d')
  \end{equation}
  which read more precisely
  ``given a vector $v \in V$ and a system of scalars
  $C(d') \geq 0$,
  if for each $d$ there exists a $d'$
  so that the first estimate holds for all $u \in V$,
  then for each $d$ there exists a $d'$
  so that the second estimate holds.''
  Indeed, choosing $d_0$ 
  so that the sum $\sum_{u \in \mathcal{B}} \mathcal{S}_{-d_0}(u)^2$,
  taken over $u$ in an orthonormal basis $\mathcal{B}$ of
  $\Delta$-eigenfunctions
  in $V$, is finite,
  and applying our hypothesis with $d + d_0$ in place of $d$,
  we obtain for some $d'$ that
  \[
  \mathcal{S}_d(v)^2
  = \sum_{u \in \mathcal{B}}
  |\langle \Delta^{d} u,  v \rangle|^2
  \leq
  C(d')^2
  \sum_{u \in \mathcal{B}} S_{-d_0}(u)^2
  \ll C(d')^2,
  \]
  as required.
  In particular, given an $\Mp_2(\mathbb{A})$-equivariant map
  $j : W \rightarrow V$,
  we have
  \begin{equation}\label{eqn:sobolev-easy-s1c}
    \langle v, j(w) \rangle \ll
    \mathcal{S}(v) \mathcal{S}(w)
    \implies
    \mathcal{S}_d(j(w)) \ll \mathcal{S}_{d'}(w).
  \end{equation}
  Here $v,w \in V,W$; $d'$ depends only upon $d$.
  (Precisely, if the estimate on the LHS holds for all $v,w$
  and some implied index $d_0$,
  then for each $d$ there exists a $d'$ so that the estimate
  on the RHS holds for all $v,w$.)
  Indeed, the hypothesis
  of   \eqref{eqn:sobolev-easy-s1c}
  is that
  $\langle v, j(w) \rangle \ll
  \mathcal{S}_{d_0}(v) \mathcal{S}_{d_0}(w)$
  for some $d_0$.
  By \eqref{eq:sobolev-duality},
  the estimate
  \[
  \langle v, j(w) \rangle
  =
  \langle \Delta^{- d - d_0} v, \Delta^{d + d_0} j(w) \rangle   
  \ll
  \mathcal{S}_{-d}(v) \mathcal{S}_{d'}(w)
  \]
  holds
  with
  $d' := d + 2 d_0$.
  By \eqref{eqn:s1c-super-basic}, we obtain
  the conclusion of \eqref{eqn:sobolev-easy-s1c}.
  The same argument gives for
  $W = L^2(\mathbf{X})$
  that
  \begin{equation}\label{eqn:sobolev-easy-s1c-X}
    \langle v, j(w) \rangle \ll
    \mathcal{S}(v) \mathcal{S}^{\mathbf{X}}(w)
    \implies
    \mathcal{S}_d(j(w)) \ll \mathcal{S}_{d'}^\mathbf{X}(w).
  \end{equation}
\item[(S1d)] \emph{Linear functionals can be bounded
    place-by-place.}  Let $\pi = \otimes \pi_\mathfrak{p}$ be a
  unitary irreducible\footnote{ Irreducibility is not mentioned
    explicitly in the hypotheses of (S1d) in
    \cite[\S2]{michel-2009}, but is used in the proof and
    necessary for the truth of the statement;
    alternatively, one could
    restrict to admissible subrepresentations of the space
    of automorphic forms.  } admissible
  representation of $\Mp_2(\mathbb{A})$.  Let
  $\ell_\mathfrak{p} : \pi_\mathfrak{p} \rightarrow \mathbb{C}$
  be functionals indexed by the places $\mathfrak{p}$ of $k$
  with the property that for all $\mathfrak{p}$ for which
  $(k_\mathfrak{p},\psi_\mathfrak{p})$ is unramified in the
  sense of \S\ref{sec:unram-case} and for which there exists a
  spherical unit vector $v_\mathfrak{p} \in \pi_\mathfrak{p}$,
  we have $|\ell(v_\mathfrak{p})| \leq 1$; assume also that
  $\ell_{\mathfrak{p}}(v_\mathfrak{p}) = 1$ for all
  $\mathfrak{p}$ outside some finite set.  Let
  $\ell = \prod \ell_\mathfrak{p} : \pi \rightarrow \mathbb{C}$
  be the restricted product of these functionals.  Suppose for
  some $A > 0$ and $d \in \mathbb{Z}$ that
  $|\ell_\mathfrak{p}| \leq A \mathcal{S}_d^{\pi_\mathfrak{p}}$
  holds for all $\mathfrak{p}$.  Then
  $|\ell| \leq A ' \mathcal{S}_{d'}^{\pi}$, where $A',d'$ depend
  only upon $A,d$.  In particular, a product of implied
  constants coming from the finite set of places at which a
  vector $v \in \pi$ ramifies can always be absorbed into
  $\mathcal{S}_{d}^\pi(v)$ if $d$ is large enough.  

  This axiom applies also to multilinear forms.  For
  instance, if
  $\ell = \prod \ell_\mathfrak{p} : \pi \otimes
  \overline{\pi} \rightarrow \mathbb{C}$
  has the property that $\ell_\mathfrak{p}$
  is bounded in magnitude by $1$ on spherical unit vectors in
  the unramified case and satisfies
  $|\ell_\mathfrak{p}(v_{1,\mathfrak{p}}, v_{2,\mathfrak{p}})|
  \leq A \mathcal{S}_{d}(v_{1,\mathfrak{p}})
  \mathcal{S}_{d}(v_{2,\mathfrak{p}})$
  in general, then
  $|\ell(v_1,v_2)| \leq A' \mathcal{S}_{d'}(v_1)
  \mathcal{S}_{d'}(v_2)$
  with notation as above; see \cite[Remark 2.6.3]{michel-2009}
  and
  \cite[\S4.4.1]{michel-2009} for details.
\item [(S2a)] \emph{Automorphic Sobolev inequality.}
  There exists $d_0$ so that
  $\mathcal{S}_{d_0}^{\mathbf{X}}$ majorizes $L^\infty$-norms.
\end{enumerate}

\subsection{The basic Weil representation and elementary theta functions}
\label{sec-4-3}
Denote by $\omega_\psi$
the Weil representation
of
$\Mp_2(\mathbb{A})$
on the Schwartz--Bruhat space $\mathcal{S}(\mathbb{A})$
underlying the dual pair $\Mp_2 \times
\O_1$ and with respect to the additive character $\psi$;
to dispel any ambiguity, we record that
$\omega_\psi(n(b)) \phi(x) = \psi(b x^2) \phi(b)$
for $\phi \in \omega$ and $b \in \mathbb{A}$.
It is the restricted tensor product of Weil representations
of the local metaplectic groups defined in \S\ref{sec:weil-rep-local}.  
Write $\omega_\psi^2 := \omega_\psi \otimes
\overline{\omega_\psi}$.

For $\phi \in \omega_\psi$,
the corresponding elementary theta function
$\theta_{\psi,\phi} : [\Mp_2] \rightarrow \mathbb{C}$
is defined by the absolutely convergent sum
\[
\theta_{\psi,\phi}(\sigma) := \sum_{\alpha \in k}
\omega_{\psi}(\sigma) \phi(\alpha).
\]
Although we have defined $\theta_{\psi,\phi}$
as a function on $[\Mp_2]$,
it is perhaps
more natural to regard it here
as the restriction
of a theta kernel
to elements of the form $(\sigma,1)$
in the product $[\Mp_2] \times O_1(k) \backslash
O_1(\mathbb{A})$.

Set $\omega_\psi^{(+)} := \{\phi \in \omega : \phi(x) = \phi(-x)
\text{ for all } x \in \mathbb{A} \}$.
Its orthogonal complement is known to be the kernel of $\phi
\mapsto \theta_{\psi,\phi}$.
Except in \S\ref{sec:anisotropic-case},
we consider only one value of $\psi$;
we accordingly abbreviate $\omega := \omega_{\psi}$,
$\omega^2 := \omega_{\psi}^2$,
$\theta_\phi := \theta_{\psi,\phi}$, etc.

By computing Fourier expansions on a Siegel domain, one finds that
$|\theta(\sigma)| \ll \htt(\sigma)^{1/4}$, 
so if $\varphi$ is a measurable function on $\mathbf{X}$ satisfying
\begin{equation}\label{eq:growth-estimate-spectral-theorem-theta}
  \varphi(\sigma) \ll \htt(\sigma)^{1/2-\delta}
  \text{ for some fixed $\delta > 0$,}
\end{equation}
then the integral $\int_{\mathbf{X}} \theta_{\phi_1} \overline{\theta_{\phi_2}} \varphi$
converges absolutely for all $\phi_1,\phi_2 \in \omega$.

{\bf Example. }
For $k := \mathbb{Q}$, $\psi_\infty(x) = e^{2 \pi i x}$ and $\phi = \otimes \phi_v$
with $\phi_\infty(x) := e^{- 2 \pi x^2}$,
$\phi_p := 1_{\mathbb{Z}_p}$, one has for $x,y \in \mathbb{R},
\mathbb{R}_+^\times$
with $z := x + i y$ that
$\theta_{\phi}(n(x) t(y^{1/2})) = y^{1/4} \sum_{n \in
  \mathbb{Z}} 
\exp(2 \pi i n^2 z)$.

\subsection{Mellin transforms and Tate integrals}
\label{sec-4-4}
Recall that we have fixed a Haar measure
$d^\times y$ on $\mathbb{A}^\times$.  
It defines a quotient
measure, which we also denote by
$d^\times y$, on $\mathbb{A}^\times/k^\times$, hence a dual
measure $d \chi$
on the space of
characters
$\chi : \mathbb{A}^\times / k^\times \rightarrow
\mathbb{C}^{\times}$
of given real part $\Re(\chi) = c$ (defined as in \S\ref{sec:conventions-characters}), so that the Mellin inversion formula
$f(1) = \int_{(c)} f^\wedge(\chi) \, d \chi$ holds for all
$f \in C_c^\infty(\mathbb{A}^\times/k^\times)$ with the Mellin transform defined by
$f^\wedge(\chi) := \int_{y \in \mathbb{A}^\times/k^\times} f(y)
y^{-\chi} \, d^\times y$.

We summarize here some standard consequences of the theory of Tate integrals (see \cite{MR0217026, MR1680912}).
For a Schwartz--Bruhat function
$\phi \in \mathcal{S}(\mathbb{A})$
and a character $\chi$ of $\mathbb{A}^\times/k^\times$
with $\Re(\chi) > 1$,
the integral
$\int_{y \in \mathbb{A}^\times} y^\chi \phi(y) \, d^\times y$
converges absolutely
for $\Re(\chi) > 1$
and extends meromorphically
to all $\chi$.
We denote by
$\int_{y \in \mathbb{A}^\times}^{\reg} y^\chi \phi(y) \,
d^\times y$
that meromorphic extension.
The possible poles are at $\chi = \alpha$ and $\chi = 0$.
One has the global Tate functional equation
\begin{equation}\label{eq:functional-equation}
  \int_{y \in \mathbb{A}^\times}^{\reg} y^\chi \phi(y) \,
  d^\times y
  = 
  \int_{y \in \mathbb{A}^\times}^{\reg} y^{\alpha-\chi} \phi^\wedge(y) \,
  d^\times y
\end{equation}
where $\phi^\wedge$ denotes the Fourier transform
with respect to any non-trivial additive character of
$\mathbb{A}/k$, such as  $\psi$ or $\psi_2$;
it does not matter which.
\subsection{Induced representations and Eisenstein series}
\label{sec-4-5}
For each character
$\chi : \mathbb{A}^\times \rightarrow
\mathbb{C}^{(1)}$ whose local components have real part $\geq 0$,
denote by $\mathcal{I}(\chi)$ its unitary induction
to $\SL_2(\mathbb{A})$,
which consists of smooth functions $f : \SL_2(\mathbb{A}) \rightarrow
\mathbb{C}$
satisfying
$f(n(x) t(y)
\sigma) = y^{\chi + \alpha} f(\sigma)$  for $x,y,\sigma \in
\mathbb{A}, \mathbb{A}^\times,\SL_2(\mathbb{A})$;
it is the restricted tensor product of the representations defined in \S\ref{sec:local-induced-reps}.
When $\chi$ is trivial on $k^\times$, so that it defines an automorphic unitary character $\chi: \mathbb{A}^\times / k^\times \rightarrow \mathbb{C}^{(1)}$, 
denote by $\Eis_\chi : \mathcal{I}(\chi) \rightarrow
\mathcal{A}(\SL_2)$,
or simply $\Eis := \Eis_\chi$ when $\chi$ is clear
from context,
the standard Eisenstein
intertwiner
obtained by averaging over
$P(k) \backslash \SL_2(k)$ and analytic
continuation
along holomorphic sections
(see e.g. \cite{MR546600}).  When $\chi$ is unitary, so is $\mathcal{I}(\chi)$,
and we equip it with the product of the invariant norms defined
in \S\ref{sec:local-induced-reps}.

As in \S\ref{sec:local-induced-reps}, the representation
$\mathcal{I}(\chi)$ is reducible when $\chi$ is a nontrivial
quadratic character; the results of \S\ref{sec-4-2}
nevertheless apply, either by continuity from the irreducible
case or by inspection of the proofs.

The Eisenstein intertwiner $\Eis_\chi$
has a simple zero for $\chi$ the trivial character,
so the normalized variant
$L(\chi,1) \Eis_\chi$ makes sense
for any unitary $\chi$.

\begin{lemma}\cite[\S2.5.1]{michel-2009}\label{lem:absolute-convergence}
  There exists $d_0$ so that $\int_{(0)} C_{\Sob}(\mathcal{I}(\chi))^{-d_0} < \infty$.
\end{lemma}

\subsection{The residue of the Eisenstein intertwiner}
The residue of the
association
$\chi \mapsto \Eis_\chi : \mathcal{I}(\chi) \rightarrow \mathcal{A}(\SL_2)$
as $\chi$ approaches the character $\alpha = |.|^1$
is given by
``integration over
$P(\mathbb{A}) \backslash \SL_2(\mathbb{A})$'' in the following
sense (see e.g. \cite{MR546600}):\footnote{ Strictly speaking,
  the cited reference discusses Eisenstein series on $\GL_2$.  The
  methods give what we state here for $\SL_2$:
  it suffices to prove the identity after testing
  both sides against an incomplete Eisenstein series $\Eis(h)$,
  $h \in C_c^\infty(N(\mathbb{A}) P(k) \backslash
  \SL_2(\mathbb{A}))$,
  and follows in that case by unfolding the summation
  defining $h$ and computing the residue of the standard
  intertwining
  operator on $\mathcal{I}(\chi)$.
}
Let
$f_\chi \in \mathcal{I}(\chi)$ be a holomorphic family
defined
in a vertical strip containing the character $\chi = \alpha$.
Suppose also that $f_\chi$ has sufficient decay as
$C(\chi) \rightarrow \infty$.  Then
for $\sigma \in \SL_2(\mathbb{A})$,
\begin{equation}\label{eq:residue-of-eisenstein-intertwiner}
  \int_{(1+\eps)} \Eis(f_\chi)(\sigma) \, d \chi
  = 
  \int_{(1-\eps)} \Eis(f_\chi)(\sigma) \, d \chi
  +
  \int_{P(\mathbb{A}) \backslash \SL_2(\mathbb{A})} f_\alpha,
\end{equation}
where $\int_{P(\mathbb{A}) \backslash \SL_2(\mathbb{A})}$
denotes the equivariant functional
$\mathcal{I}_\chi(\alpha) \rightarrow \mathbb{C}$
compatible with the chosen Haar measures on $P(\mathbb{A})$ and $\SL_2(\mathbb{A})$.
As a ``dimensionality test,'' note that both
$d \chi$ and $\int_{P(\mathbb{A}) \backslash \SL_2(\mathbb{A})}$ scale inversely
with respect to the measure $d^\times y$ on $\mathbb{A}^\times$.

\subsection{Bounds for the Eisenstein intertwiner}
% Recall from \S\ref{sec:intro-main-results} that
%  $\Eis_\chi^* := \Lambda(\chi,1) \Eis_\chi$.
\begin{lemma}\label{lem:error-bound-eis-star}
  Let $\chi$ be a unitary character of
  $\mathbb{A}^\times/k^\times$.
  For $f \in \mathcal{I}(\chi)$ and $\varphi \in \mathcal{A}(\SL_2)$,
  \[
  L(\chi,1) \langle \Eis(f), \varphi \rangle_{L^2(\mathbf{X})}
  \ll \mathcal{S}(f) \|\varphi\|_{\infty}.
  \]
\end{lemma}
\begin{proof}
  It suffices to estimate the integral of $|L(\chi,1) \Eis(f)
  \varphi|$ over a Siegel domain.
  For $\chi$ close to the trivial character (i.e.,
  near the pole of $L(\chi,1)$),
  we estimate the constant term of $L(\chi,1) \Eis(f)$
  as in the proof of \cite[(4.12)]{michel-2009}
  and its Whittaker function
  using the argument of \cite[(3.23)]{michel-2009}
  and that linear functionals
  can be bounded place-by-place
  (S1d),
  giving
  \begin{equation}\label{eq:eis-on-siegel-estimate}
        L(\chi,1) \Eis(f)(g) \ll
    \mathcal{S}(f)
    \htt(g)^{1/2} \log(\htt(g) + 10).
  \end{equation}
  For $\chi$ away from the trivial character,
  we first argue similarly
  that
  $\Eis(f)(g) \ll  \mathcal{S}(f)
  \htt(g)^{1/2}$
  and then use the coarse bound
  $L(\chi,1) \ll C(\chi)^{O(1)}
  \ll C_{\Sob}(\mathcal{I}(\chi))^{O(1)}$
  as in \cite[\S4.1.8]{michel-2009}
  to absorb the dependence upon $\chi$ into the factor
  $\mathcal{S}(f)$ at the cost of increasing its implied index.
  Thus \eqref{eq:eis-on-siegel-estimate} holds for any $\chi$,
  and so
  \[
    L(\chi,1) \langle \Eis(f), \varphi \rangle_{L^2(\mathbf{X})}
    \ll
    \mathcal{S}(f)
    \int_{\mathbf{X}} \htt(g)^{1/2} \log(\htt(g) + 10) |\varphi|(g).
  \]
  Since
  $\int_{\mathbf{X}} \htt(g)^{1/2} \log(\htt(g) + 10) < \infty$,
  the conclusion follows.
\end{proof}

\subsection{Bounds for the Eisenstein projector}
For $\varphi \in L^\infty(\mathbf{X})$ there exists, by duality,
an element
$\Pi_\chi(\varphi) \in \mathcal{I}(\chi)$ so that for all $f \in \mathcal{I}(\chi)$,
\begin{equation}\label{eq:duality-eis-star}
  \langle f, \Pi_{\chi}(\varphi) \rangle_{\mathcal{I}(\chi)} =
  \langle \Eis(f), \varphi \rangle_{L^2(\mathbf{X})}.
\end{equation}
The map $\Pi_\chi$ is linear and equivariant.
By continuity and the discussion of \S\ref{sec-4-5},
we may consider $L(\chi,1) \Pi_\chi$
even when $\chi =0$.
\begin{lemma}\label{lem:error-bound-proj-star-phi}
  Let $\chi$ be a unitary character of
  $\mathbb{A}^\times/k^\times$.
  For any $d$ there exists $d'$ so that
  \[
  L(\chi,1) \mathcal{S}_d(\Pi_\chi(\varphi)) \ll
  \mathcal{S}_{d'}^{\mathbf{X}}(\varphi).
  \]
\end{lemma}
\begin{proof}
  By \eqref{eq:duality-eis-star},
  Lemma \ref{lem:error-bound-eis-star} and the automorphic Sobolev inequality (S2a),
  we have
  $\langle f, \Pi_\chi(\varphi) \rangle \ll \mathcal{S}(f)
  \mathcal{S}^{\mathbf{X}}(\varphi)$.
  The conclusion follows from Sobolev embedding (S1c) in the form
  \eqref{eqn:sobolev-easy-s1c-X}.
\end{proof}

\subsection{Change of polarization}
\label{sec-4-6}
Recall that $\omega^2  := \omega \otimes \overline{\omega}$;
it descends to a representation of $\SL_2(\mathbb{A})$ on
$\mathcal{S}(\mathbb{A})
\otimes \mathcal{S}(\mathbb{A}) \cong \mathcal{S}(\mathbb{A}^2)$.
Define the  partial Fourier transform $\mathcal{F} : \omega^2
\rightarrow \mathcal{S}(\mathbb{A}^2)$
by taking the restricted tensor product of the local maps
defined in \S\ref{sec-3-11},
thus
$\mathcal{F} \phi(y_1,y_2)
= 
\int_{t \in \mathbb{A}}
\phi(\frac{y_1 + t}{2}, \frac{y_1 - t}{2})
\psi(y_2 t)$.
By Lemma \ref{lem:local-change-polarization},
\begin{equation}\label{eq:partial-fourier-transform-equivariance}
  \mathcal{F}(\omega^2(\sigma) \phi)(x) =
  \mathcal{F} \phi(x \sigma).
\end{equation}
For $\phi = \phi_1 \otimes \overline{\phi_2}$,
Fourier inversion gives
$\mathcal{F} \phi(0) = \int_{x \in \mathbb{A}} \phi_1(x)
\overline{\phi_2}(x)$
and
$\int_{\mathbb{A}^2} \mathcal{F} \phi
= \int_{x \in \mathbb{A}} \phi_1(x) \overline{\phi_2}(-x)$,
whence
\begin{equation}\label{eq:sum-of-phi-stuff}
  \mathcal{F} \phi(0)
  + \int_{\mathbb{A}^2} \mathcal{F} \phi
  =
  2
  \langle \phi_1, \phi_2 \rangle
  \text{ for }
  \phi_1, \phi_2 \in \omega^{(+)}.
\end{equation}

% \subsection{The adelic intertwiner\label{sec:adelic-intertwiner}}
% \label{sec-4-7}
\subsection{The regularized global intertwiner}
\label{sec-4-8}
Let
$\chi : \mathbb{A}^\times/k^\times \rightarrow
\mathbb{C}^\times$
be a nontrivial Hecke character
with $\Re(\chi) > -1$.
We
now define an $\Mp_2(\mathbb{A})$-equivariant map
$I_\chi : \omega \otimes \overline{\omega} \rightarrow
\mathcal{I}(\chi)$.
It may be characterized
most simply
as the unique equivariant map
for which
\begin{equation}\label{eq:symmetric-characterization-I-chi}
  I_\chi(\phi_1 \otimes \phi_2)(1)
  =
  \int_{y \in \mathbb{A}^\times}^{\reg}
  y^{-\chi}
  \phi_1(y) \phi_2(y).
\end{equation}
By the global Tate functional equation
\eqref{eq:functional-equation} and the identity
\eqref{eq:partial-fourier-transform-equivariance},
we may
recast this definition
in the following form, which is better
suited for our purposes:
\begin{equation}\label{eq:regularized-projector-functional-equation}
  I_\chi(\phi)(\sigma)
  :=
  \int_{y \in \mathbb{A}^\times}^{\reg}
  y^{\alpha + \chi} \mathcal{F} \phi(y e_2 \sigma)
\end{equation}
Equivalently,
$I_\chi$ is
the restricted tensor product of the local maps defined in
\S\ref{sec-3-12},
regularized by the local factors $L(\chi_\mathfrak{p},1)$: for pure tensors
$\phi = \otimes \phi_{\mathfrak{p}}$
% _1 = \otimes \phi_{1,{\mathfrak{p}}}, \phi_2 = \otimes
% \phi_{2,{\mathfrak{p}}}$  
% in $\omega = \otimes \omega_{\mathfrak{p}}$
% with local tensor products
% $\phi_{\mathfrak{p}} := \phi_{1,{\mathfrak{p}}} \otimes
% \overline{\phi_{2,{\mathfrak{p}}}}$,
and a good factorization
$d^\times y = \prod d^\times y_{\mathfrak{p}}$,
one has
by
Lemma
\ref{lem:local-intertwiner-unramified-case}
\[
  I_\chi(\phi)(\sigma)
  =
  \Lambda(\chi,1)
  \prod_{\mathfrak{p}}
  \frac{I_{\chi_\mathfrak{p}}(\phi_{\mathfrak{p}})
  }{
    L(\chi_\mathfrak{p},1)
  }
  =
  L^{(S)}(\chi,1)
  \prod_{\mathfrak{p} \in S} I_{\chi_\mathfrak{p}}(\phi_{\mathfrak{p}})
  % \frac{\int_{y_{\mathfrak{p}}  \in k_\mathfrak{p}^\times}
  % y_{\mathfrak{p}}^{\alpha + \chi_\mathfrak{p}} \mathcal{F} \phi_{\mathfrak{p}}(y_{\mathfrak{p}}  e_2 \sigma_{\mathfrak{p}}) \, d^\times y_{\mathfrak{p}}
  % }{
  %   L(\chi_\mathfrak{p} ,1)
  % }.
\]
for $S$ a finite set of places
taken large enough in terms of $\phi_1, \phi_2, \sigma$.
In particular, the map $\chi \mapsto L(\chi,1)^{-1} I_\chi(\phi)$
extends by continuity to all unitary $\chi$;
similarly,
$\Eis_\chi(I_\chi(\phi))$ makes sense for any unitary $\chi$.

\begin{lemma}\label{lem:error-bound-regularized-transform-phi}
  For any $d, d_0$ there exists $d'$ so that
  for all $\phi = \phi_1 \otimes \overline{\phi_2} \in \omega
  \otimes \overline{\omega}$ and all 
  nontrivial unitary $\chi : \mathbb{A}^\times \rightarrow \mathbb{C}^{(1)}$,
  one  has
  \[
    L(\chi,1)^{-1}
    \mathcal{S}_d(I_\chi(\phi))
    \ll \mathcal{S}_{d'}(\phi_1)\mathcal{S}_{d'}(\phi_2)
    C_{\Sob}(\mathcal{I}(\chi))^{-d_0}.
  \]
\end{lemma}
\begin{proof}
  The definition \eqref{eq:sobolev-conductor-defn} of  $C_{\Sob}$
  implies that $\mathcal{S}_d(I_\chi(\phi)) \ll \mathcal{S}_{d +
    d_0}(I_\chi(\phi))
  C_{\Sob}(\mathcal{I}(\chi))^{-d_0}$, so it suffices to show that
  for any $d$ there exists
  $d'$ so that $L(\chi,1)^{-1} \mathcal{S}_d(I_\chi(\phi)) \ll
  \mathcal{S}_{d'}(\phi_1)\mathcal{S}_{d'}(\phi_2)$.  By Sobolev
  embedding (S1c) we reduce to showing that for each
  $d$ there exists $d'$ so that for each factorizable $f =
  \otimes f_\mathfrak{p} \in
  \mathcal{I}(\chi)$, one has
  \[
    L(\chi,1)^{-1} \langle f, I_\chi(\phi) \rangle
    \ll \mathcal{S}_{-d}(f) \mathcal{S}_{d'}(\phi_1)
    \mathcal{S}_{d'}(\phi_2).
  \]
  This follows from the fact that
  linear functionals can be bounded place-by-place (S1d) applied
  to a suitable multiple of the factorizable functional $\ell =
  \prod \ell_\mathfrak{p}$ defined by $\ell(\phi) := \langle f,
  I_\chi(\phi) \rangle / (L(\chi,1) \mathcal{S}_{-d}(f))$, the required local bounds following from
  Proposition \ref{prop:local-sobolev-estimate}, Lemma
  \ref{lem:local-intertwiner-unramified-case}, the duality
  \eqref{eq:sobolev-duality}, and the estimate
  $L(\chi_\mathfrak{p},1) \asymp 1$ for the local factors at
  finite places $\mathfrak{p}$;
  compare with the proof of
  \cite[(4.25)]{michel-2009}.
\end{proof}

% Let
% $\chi : \mathbb{A}^\times/k^\times \rightarrow
% \mathbb{C}^\times$
% be an automorphic character.
% For $\chi \neq 0$,
% define an $\Mp_2(\mathbb{A})$-equivariant map
% $I_\chi^* : \omega \otimes \overline{\omega} \rightarrow
% \mathcal{I}(\chi)$
% by requiring that
% \[
%   I_\chi^*(\phi_1,\phi_2)(1)
%   =
%   \int_{y \in \mathbb{A}^\times}^{\reg}
%   y^{-\chi}
%   \phi_1(y) \overline{\phi_2}(y).
% \]
% and that
% $I_\chi^*(\phi_1,\phi_2)(\sigma)
% = I_\chi^*(\omega(\sigma)\phi_1,\omega(\sigma)\phi_2)(1)$
% for all $\sigma \in \Mp_2(\mathbb{A})$.
% By the global Tate functional equation
% \eqref{eq:functional-equation} and the identity
% \eqref{eq:partial-fourier-transform-equivariance}, we obtain with the notation $\phi := \phi_1 \otimes \overline{\phi_2}$ 
% the integral representation (absolutely convergent for $\Re(\chi) > 0$)
% \begin{equation}\label{eq:regularized-projector-functional-equation}
%   I_\chi^*(\phi_1,\phi_2)(\sigma)
%   =
%   \int_{y \in \mathbb{A}^\times}^{\reg}
%   y^{\alpha + \chi} \mathcal{F} \phi(y e_2 \sigma).
% \end{equation}
% Consequently,
% \begin{equation}\label{eqn:I-chi-star-I-chi}
%   I_\chi^* = \Lambda(\chi,1) I_\chi.
% \end{equation}

\subsection{The anisotropic case}
\label{sec:anisotropic-case}
Let $\psi, \psi ' : \mathbb{A}/k \rightarrow \mathbb{C}^{(1)}$
be nontrivial characters.
There exists $a \in k^\times$
so that $\psi ' (x) = \psi (a x)$.
Assume that $a \notin k^{\times 2}$.
Then $\omega_{\psi} \not\cong \omega_{\psi '}$.
Denote by $(V,q)$ the quadratic space $k^2$
with the form $q(x,y) := x^2 - a y^2$;
it
is non-split.
Then
\begin{equation}\label{eq:decomp-nonsplit}
  \omega_{\psi} \otimes
  \overline{\omega_{\psi '}} \cong \omega_{\psi,V}.
\end{equation}
Denote by $\SO(V)$ be the special orthogonal group of $V$.
The quotient
\[
  [\SO(V)] := \SO(V)(k) \backslash \SO(V)(\mathbb{A})\]
is then a
compact abelian group; equip it with the probability Haar.
One has a Weil representation
$\omega_{\psi,V}$ of $\SL_2(\mathbb{A})$
on $\mathcal{S}(V(\mathbb{A}))$ as in \S\ref{sec:weil-rep-local}, \S\ref{sec-4-3}.
The constructions to follow depend upon $\psi$,
but we omit that dependence from our notation for clarity.
For $\phi \in \omega_{\psi,V}$,
denote by
$\theta_\phi: \mathbf{X} \times [\SO(V)] \rightarrow \mathbb{C}$
the theta kernel
$\theta_{\phi}(\sigma,h) := \sum_{\delta \in V}
\omega_{\psi,V}(\sigma) \phi(h^{-1} \delta)$.
Denote by $A(\SO(V))$ the set of characters $\tau$ of
$[\SO(V)]$; they are all unitary.
For each $\tau \in A(\SO(V))$,
denote by $\theta_{\phi,\tau} : \mathbf{X} \rightarrow \mathbb{C}$
the theta function
$\theta_{\phi,\tau}(\sigma) := \int_{h \in [\SO(V)]}
\tau^{-1}(h) \theta_{\phi}(\sigma,h)$
and by
$\theta(\tau)$
the theta lift
$\theta(\tau) := \{\theta_{\phi,\tau} : \phi \in
\mathcal{S}(V(\mathbb{A}))\}$;
it defines an irreducible automorphic ``dihedral'' representation of
$\SL_2(\mathbb{A})$ which is abstractly unitarizable.
Fix the following unitary structure on $\theta(\tau)$:
\begin{itemize}
\item If
  $\tau \neq 1$, then $\theta(\tau)$ is cuspidal; equip it with
  the unitary structure compatible with its inclusion into
  $L^2(\mathbf{X})$.
\item   If $\tau = \mathbf{1}$ is the trivial
  character of $[\SO(V)]$, then the Siegel--Weil
  formula (see
  \cite[p182]{MR783511} or
  \cite{2012arXiv1207.4709T}  and references)
  implies
  that $\theta(\mathbf{1}) \subseteq \Eis(\mathcal{I}(\eta))$ where $\eta$
  is the (unitary) quadratic character of
  $\mathbb{A}^\times / k^\times$ corresponding under class field
  theory to the quadratic field extension $k(\sqrt{a})/k$.  More
  precisely, one has for each $\phi \in \omega_{\psi,V}$ the
  identity
  $\theta_{\phi,\mathbf{1}}(\sigma) = \frac{1}{2}
  \Eis(J_\eta(\phi))(\sigma)$,
  where $\Eis : \mathcal{I}(\eta) \rightarrow \mathcal{A}(\SL_2)$
  is as in \S\ref{sec-4-5} and
  $J_\eta : \omega_{\psi,V} \rightarrow \mathcal{I}(\eta)$ is
  given by
  $J_\eta(\phi)(\sigma) := \omega_{\psi,V}(\sigma) \phi(0)$.
  Equip $\theta(\mathbf{1})$ with the unitary structure coming
  from $\mathcal{I}(\eta)$.
\end{itemize}
Equip $\theta(\tau)$ with unitary Sobolev norms $\mathcal{S}_d^{\theta(\tau)}$
(\S\ref{sec-4-2}).
As in \S\ref{sec-4-5}, there are linear maps
$\Pi_{\theta(\tau)}: L^\infty(\mathbf{X}) \rightarrow
\theta(\tau)$
such that
$\langle f, \Pi(\varphi) \rangle_{\theta(\tau)} = \langle f,
\varphi \rangle_{L^2(\mathbf{X})}$
for all $f \in \theta(\tau), \varphi \in L^\infty(\mathbf{X})$.
(The precise choice
$L^\infty(\mathbf{X})$
of domain is not particularly important.)
For
$\phi_1, \phi_2 \in \omega_{\psi}, \omega_{\psi '}$
we have by the isomorphism \eqref{eq:decomp-nonsplit}
and Fourier inversion that
\[
\theta_{\phi_1}(\sigma) \overline{\theta_{\phi_2}(\sigma)} =
\theta_{\phi}(\sigma,1)
= \sum_{\tau \in A(\SO(V))} \theta_{\phi,\tau}(\sigma).
\]
\begin{lemma}\label{lem:anisotropic} Let $\varphi \in L^\infty(\mathbf{X})$ and notation as above.
  \begin{enumerate}
  \item[(i)] For $\tau \neq 1$, one has
    $\mathcal{S}_d^{\theta(\tau)}(\Pi_{\theta(\tau)}(\varphi)) \leq
    \mathcal{S}_d(\varphi)$.
  \item[(ii)]
    For each $d$ there exists $d'$ so that
    $\mathcal{S}_d^{\theta(\mathbf{1})}(\Pi_{\theta(\mathbf{1})}(\varphi))
    \ll \mathcal{S}_{d'}^{\mathbf{X}}(\varphi)$.
  \item[(iii)]
    There exists $d_0$ so that
    $\sum_{\tau \in A(\SO(V))}
    C_{\Sob}(\theta(\tau))^{-d_0} \ll 1$.
  \item[(iv)]
    For any $d, d_0$ there exists $d'$
    so that for all $\tau$
    and all $\phi_1,\phi_2 \in \omega_{\psi}, \omega_{\psi '}$,
    one has
    \[
    \mathcal{S}_d^{\theta(\tau)}(\theta_\tau(\phi_1 \otimes
    \overline{\phi_2})) \ll
    \mathcal{S}_{d'}(\phi_1) \mathcal{S}_{d'}(\phi_2)
    C_{\Sob}(\tau)^{-d_0}.
    \]
  \end{enumerate}
  The implied constants are allowed to depend upon $k,\psi,\psi '$ and hence upon $V$.
\end{lemma}
\begin{proof}
  (i): Immediate from the definitions and normalizations of
  unitary structures.  (ii): Repeat the proof of the case
  $\chi = \eta$ of Lemma \ref{lem:error-bound-proj-star-phi}.
  (iii): Follows in a
  stronger form from \cite[\S2.5.1]{michel-2009}.  (iv) When
  $\tau = \mathbf{1}$, this follows from Lemma
  \ref{lem:error-bound-regularized-transform-phi}.  A similar
  proof applies in the case $\tau \neq \mathbf{1}$;
  we sketch it for completeness.  First, reduce formally as
  in the proofs of Lemmas
  \ref{lem:error-bound-regularized-transform-phi} and
  \ref{prop:local-sobolev-estimate} to the case $d_0 = d = 0$,
  then
  by reduction theory (see \S\ref{sec:siegel-domains})
  to showing for all
  $x,y \in \mathbb{A}, \mathbb{A}^\times$  for which
  $|y| \gg 1$ and all $k$
  in a fixed compact subset of $\SL_2(\mathbb{A})$
  that
  $\phi := \phi_1 \otimes \overline{\phi_2}$
  satisfies
  $\theta_{\phi,\tau}(n(x) t(y) k) \ll \mathcal{S}(\phi_1)
  \mathcal{S}(\phi_2)$,
  then by the distortion property (\S\ref{sec:local-sobolev-norms}) and equivariance to
  the case $k = 1$.  By definition,
    \[\theta_{\phi,\tau}(n(x) t(y))
    =
    \int_{h \in [\SO(V)]}
    \tau^{-1}(h)
    |y|
    \eta(y)
    \sum_{\delta \in V}
    \psi(x q(\delta))
    \phi(h^{-1} y \delta).
    \]
    Since $\tau \neq \mathbf{1}$, the inner sum may be restricted to $\delta \neq 0$.
    Since $[\SO(V)]$ is compact,
    we reduce
    to showing for all $y \in \mathbb{A}^\times$ with $|y| \gg
    1$
    that
    \[
    |y|
    \sum_{\delta \in V - \{0\}}
    |\phi(h^{-1} y \delta)| \ll \mathcal{S}(\phi_1) \mathcal{S}(\phi_2),
    \]
    uniformly for $h$ in a compact subset of
    $\SO(V)(\mathbb{A})$.  We reduce using (S1d) to the case of
    pure tensors $\phi_{i} = \otimes \phi_{i,\mathfrak{p}}$ and
    then by (S4d) to the case that each
    $\phi_{i,\mathfrak{p}}$ is a
    $\Delta_\mathfrak{p}$-eigenfunction.  We estimate the size and support
    of the non-archimedean components and decay of the
    archimedean components as in
    the proof of Lemma \ref{lem:local-bound-1}.
    In the number field case, we then
    modify $y$ by a suitable element of $k^\times$,
    using the compactness of $\mathbb{A}^{(1)} / k^\times$,
    to reduce
    to showing: for $L$ a fixed lattice in $V$,
    the quantity
    $\# \{\delta \in h L : \|q(\delta)\| \leq Q \} - 1$
    vanishes for $Q$ small enough
    and is otherwise $O(Q^{O(1)})$
    for all $h$ in a fixed compact subset of
    $\SO(V)(\mathbb{A})$, where $\|q(\delta)\|$ denotes the
    maximum of the norms coming from the archimedean
    completions.
    The function field case is similar.
\end{proof}

\subsection{The $\Xi$-function}
\label{sec-4-9}
The Harish--Chandra function
$\Xi : \SL_2(\mathbb{A}) \rightarrow \mathbb{R}_{>0}$ is the
matrix coefficient of the normalized spherical vector in $\mathcal{I}(0)$.
It controls the matrix coefficients of tempered
representations in the following sense:
\begin{lemma}\cite[\S2.5.1]{michel-2009}\label{lem:error-basic-matrix-coef-bound}
  Let $\pi$ be a tempered unitary representation of $\SL_2(\mathbb{A})$.
  For $f_1, f_2 \in \pi$
  and $\sigma \in \SL_2(\mathbb{A})$, one has
  $\langle f_1, g f_2 \rangle \ll \Xi(\sigma) \mathcal{S}(f_1) \mathcal{S}(f_2)$.
\end{lemma}
Lemma \ref{lem:error-basic-matrix-coef-bound} applies in particular to $\pi = \mathcal{I}(\chi)$
for any unitary character $\chi$ of
$\mathbb{A}^\times / k^\times$ or to any of the non-split
dihedral theta lifts
$\pi = \theta(\tau)$ as in \S\ref{sec:anisotropic-case}.
For orientation, we record that $\Xi$ factors as $\Xi(\sigma) = \prod_{\mathfrak{p}}
\Xi_\mathfrak{p}(\sigma_\mathfrak{p})$,
is left and right invariant under the standard maximal compact,
and is given locally at a finite place $\mathfrak{p}$ with
uniformizer $\varpi$
and $|\varpi|^{-1} = q$
for $m \geq 0$
by
$\Xi_\mathfrak{p}(t(\varpi^m)) = (2 m + 1)/q^m - 1_{m>0} (2 m - 1) / q^{m+1} \asymp (2 m + 1) / q^m$.

\section{Proofs of the theorems}
\label{sec-5}

\begin{proof}[Proof of Theorem \ref{thm:theta-plancherel-redux}]
  Let $\phi_1, \phi_2 \in \omega_{\psi}^{(+)}$.
  Set $\phi := \phi_1 \otimes \overline{\phi_2} \in \omega^2 := \omega_{\psi} \otimes \overline{\omega_{\psi}}$.
  Let $\sigma \in \Mp_2(\mathbb{A})$.
  By the definition followed by Poisson summation,
  \[
    \theta_{\psi,\phi_1}(\sigma) \overline{\theta_{\psi,\phi_2}}(\sigma)
    = \sum_{x \in k^2} \omega^2(\sigma) \phi(x)
    = \sum_{x \in k^2} \mathcal{F}(\omega^2(\sigma))(x).
  \]
  By \eqref{eq:partial-fourier-transform-equivariance},
  the RHS equals
  $\sum_{x \in k^2}
  \mathcal{F} \phi(x \sigma)$.
  Since every nonzero element of $k^2$ is uniquely of the form
  $t e_2 \gamma$ for some $t \in k^\times, \gamma \in P(k)
  \backslash \SL_2(k)$,
  it follows that
  \[
  \theta_{\psi,\phi_1}(\sigma) \overline{\theta_{\psi,\phi_2}}(\sigma)
  = \mathcal{F} \phi(0)
  + \Eis(f)(\sigma)
  \]
  where
  $f(\sigma) := \sum_{t \in k^\times} \mathcal{F} \phi(t e_2 \sigma)$
  and
  $\Eis(f)(\sigma) := \sum_{\gamma \in P(k) \backslash \SL_2(k)}
  f(\gamma \sigma)$ is an incomplete Eisenstein series.
  The asymptotics of $f$ with respect to the Iwasawa height (\S\ref{sec-4-1}) are analogous to
  those in \S\ref{sec:basic-idea}.
  Denote by
  $f_\chi \in \mathcal{I}(\chi)$ 
  the unitarily normalized Mellin transform
  \begin{equation}\label{eq:unfolding-normalized-mellin-transform-induced}
    f_\chi(\sigma)
    :=
    \int_{y \in \mathbb{A}^\times /k^\times}
    y^{-\chi - \alpha}
    f(t(y) \sigma) \, d^\times y
    =
    \int_{y \in \mathbb{A}^\times}
    y^{\chi + \alpha}
    \mathcal{F} \phi(y e_2 \sigma) \, d^\times y
    =
    I_\chi(\phi)(\sigma).
  \end{equation}
  By partial integration, $f_\chi(\sigma)$
  decays rapidly with respect to $\chi$,
  whence the rapidly-convergent Mellin expansion
  \begin{equation}\label{eqn:incomplete-eis-mellin-expand}\Eis(f)
    = \int_{(2)} \Eis(f_\chi) \, d \chi.
  \end{equation}
  The map $\chi \mapsto \Eis(f_\chi) = \Eis(I_\chi(\phi))$ is
  holomorphic for $\Re(\chi) \geq 0$ except for a simple pole at
  $\chi = \alpha$.  We shift contours to the line
  $\Re(\chi) = 0$; thanks to
  \eqref{eq:residue-of-eisenstein-intertwiner},
  \eqref{eq:regularized-projector-functional-equation} and our
  measure normalizations, the pole contributes
  \[
  \int_{P(\mathbb{A}) \backslash \SL_2(\mathbb{A})}
  I_\alpha(\phi)
  =
  \int_{U(\mathbb{A}) \backslash \SL_2(\mathbb{A})}
  \mathcal{F} \phi(e_2 \sigma)
  =
  \int_{\mathbb{A}^2}
  \mathcal{F} \phi.
  \]
  In summary, we have shown that
  \[
  \theta_{\psi,\phi_1}(\sigma) \overline{\theta_{\psi,\phi_2}}(\sigma)
  =
  \mathcal{F} \phi(0)
  +
  \int_{\mathbb{A}^2} \mathcal{F} \phi
  + 
  \int_{(0)} \Eis(I_\chi(\phi))(\sigma) \, d \chi.
  \]
  We conclude by \eqref{eq:sum-of-phi-stuff} and reasoning similar
  to that following \eqref{eq:toy-decomp}.
\end{proof}

\begin{proof}[Proof of Theorem \ref{thm:hecke-equid-redux}]
  By \eqref{eq:duality-eis-star} and Lemma
  \ref{lem:error-basic-matrix-coef-bound}
  followed by Lemmas
  \ref{lem:error-bound-regularized-transform-phi}
  and \ref{lem:error-bound-proj-star-phi},
  \[
  \int \Eis(I_\chi(\phi)) \cdot \sigma \varphi
  \ll 
  \Xi(\sigma)
  \mathcal{S}(I_\chi(\phi))
  \mathcal{S}(\Pi_\chi \varphi)
  \ll 
  \Xi(\sigma)
  \mathcal{S}(\phi) C_{\Sob}(\mathcal{I}(\chi))^{-d_0}
  \mathcal{S}^{\mathbf{X}}(\varphi).
  \]
  We deduce the $\psi = \psi '$ case by integrating over
  $\chi$
  and applying
  Lemma \ref{lem:absolute-convergence}
  and Theorem \ref{thm:theta-plancherel-redux}.
  The $\psi \neq \psi '$ case is proved
  similarly using Lemma \ref{lem:anisotropic}.
\end{proof}

\begin{remark}\label{rmk:sharpenings}
  A slightly lengthier argument gives a stronger
  (but more complicated, and not obviously more useful) estimate with
  $\mathcal{S}_d^{\mathbf{X}}(\varphi_0)$ replaced by
  $\|\htt^{1/2+\eps} \Delta_S^d \varphi_0\|$ for
  $S$ the finite
  set of places $\mathfrak{p}$ for which $\sigma_\mathfrak{p}$
  is not in the maximal compact and $\Delta_S$ the product of
  local Laplacians (\S\ref{sec:local-sobolev-norms}) at those
  places.
  One can also specify more precisely the dependence upon
  $\phi_1, \phi_2$.
\end{remark}

% \cite{nelson-subconvex-reduction-eisenstein}
% \cite{2017arXiv170905615M}

% -----------------
\bibliography{refs}{}
\bibliographystyle{plain}
% -----------------
\end{document}